\documentclass[11pt]{amsart}

\usepackage{palatino}
\usepackage{amsfonts}
\usepackage{amssymb}
\usepackage{amscd}
\usepackage{xypic}
\usepackage{graphicx}
\usepackage{float}
\usepackage{amsmath}

\usepackage{amsmath}
\usepackage{mathrsfs}

\usepackage{epstopdf}

\newtheorem{introtheorem}{Theorem}

\newtheorem{theorem}{Theorem}[section]
\newtheorem{proposition}[theorem]{Proposition}

\newtheorem{lemma}[theorem]{Lemma}
\theoremstyle{definition}
\newtheorem{definition}[theorem]{Definition}
\newtheorem{remark}[theorem]{Remark}

\newtheorem{conjecture/question}[theorem]{Conjecture/Question}
\newcommand{\dblq}{{/\!/}}

\newtheorem{remark/definition}[theorem]{Remark/Definition}
\newtheorem{terminology/notation}[theorem]{Terminology/Notation}

\def\PPP{{\textbf P}}

\def\GG{{\textbf G}}

\def\PP{{\textbf P}}
\def\kk{{\textbf k}}
\def\OO{\mathcal{O}}

\def\cB{\mathcal{B}}
\def\cA{\mathcal{A}}
\def\F{\mathcal{F}}

\def\E{\mathcal{E}}
\def\G{\mathcal{G}}
\def\cS{\mathcal{S}}

\def\L{\mathcal{L}}
\def\I{\mathcal{I}}

\def\cM{\mathcal{M}}

\def\cR{\mathcal{R}}
\def\rr{\overline{\mathcal{R}}}
\def\cZ{\mathcal{Z}}
\def\cU{\mathcal{U}}
\def\cQ{\mathcal{Q}}
\def\cC{\mathcal{C}}
\def\H{\mathcal{H}}
\def\HH{\mathfrak{H}}
\def\Pic0{{\rm Pic}^0(X)}

\def\mm{\overline{\mathcal{M}}}
\def\hh{\overline{\mathcal{H}}}

\def\dd2{\overline{\mathfrak{D}}_{\mathrm{rk}2}}

\begin{document}
\title{The moduli space of curves of genus $16$ is uniruled}
\author[G. Farkas]{Gavril Farkas}

\address{Humboldt-Universit\"at zu Berlin, Institut f\"ur Mathematik,  Unter den Linden 6
\hfill \newline\texttt{}
 \indent 10099 Berlin, Germany} \email{{\tt farkas@math.hu-berlin.de}}

\author[A. Verra]{Alessandro Verra}
\address{Universit\`a Roma Tre, Dipartimento di Matematica e Fisica\newline \indent  Largo San Leonardo Murialdo 
  1-00146 Roma, Italy}
 \email{{\tt alessandro.verra@mat.uniroma3.it}}

\begin{abstract}
Using in an essential way the geometry of  a non-abelian Brill--Noether divisor, we establish that $\mm_{16}$ is unruled. This is the highest genus for which this is known. 
\end{abstract}
\maketitle

\vskip 5pt

\section{Introduction}

The problem of determining the global nature of the moduli space $\mm_g$ of stable curves of genus $g$ has long been one of the key questions in the field, motivating important developments in moduli theory like the theories of admissible covers, or limit linear series. Severi \cite{Sev}  proved that $\mm_g$ is unirational for $g\leq 10$, see \cite{AC} for a modern presentation. Much later, in the celebrated series of papers \cite{HM}, \cite{H}, \cite{EH}, Harris, Mumford and Eisenbud showed that $\mm_g$ is of general type for $g\geq 24$. Recently, it has been showed in \cite{FJP} using tropical methods that both moduli spaces $\mm_{22}$ and $\mm_{23}$ are of general type. On the other hand, due to work of Sernesi \cite{Ser}, Chang-Ran \cite{CR1}, \cite{CR2} and Verra \cite{Ve} it is known that $\mm_g$ is unirational also for $11\leq g\leq 14$. Computer assisted implementations over finite fields of all these unirational constructions can be found in \cite{Sch1}. Bruno and Verra \cite{BV} proved that $\mm_{15}$ is rationally connected. An alternative proof, via matrix factorization of cubics,  of the weaker result that $\mm_{15}$ is uniruled has been found by Schreyer \cite{Sch2}. We finally mention Savin's paper \cite{Sa} on whether in positive characteristic $\mm_{g,n}$ may simultaneously be uniruled and of general type for large $g$. Our main result is the following:

\begin{introtheorem}\label{main_thm}
The moduli space $\mm_{16}$  is uniruled.
\end{introtheorem}

Note that $16$ is the highest genus for which it is currently known that $\mm_{16}$ is uniruled. We refer to Tseng's note \cite{Ts} for further details on the convoluted history of determining the Kodaira dimension of $\mm_{16}$ and to \cite{AB} for further developments prior to this paper. It is also worth noting that it is proved in \cite{F} that the generically finite cover $\hh_{16,9}$ of $\mm_{16}$ parametrizing admissible covers of (minimal) degree $9$ has non-negative Kodaira dimension. The main novelty in this paper is the decisive use of non-abelian Brill-Noether theory to questions on the birational geometry of $\mm_g$. 

\vskip 4pt

The first such divisor is the closure $\Delta_0$ of the locus in $\mm_{16}$ of irreducible $1$-nodal curves of arithmetic genus $16$. We rely on the  \emph{uniruled parametrization} of $\mm_{15}$ found by Bruno and Verra \cite{BV}. Their work establishes that through a general point of $\mm_{15}$ there passes not only a rational curve, but in fact a \emph{rational surface}. This extra degree of freedom yields a uniruled parametrization of $\mm_{15,2}$. Since one has a degree $2$ map $\mm_{15,2}\rightarrow \Delta_0$, this induces a uniruled parametrization of  $\Delta_0$. We show the following:

\begin{introtheorem}\label{sweep}
The boundary divisor $\Delta_0$ of $\mm_{16}$ is uniruled and swept by a family of rational curves, whose general member $\Gamma\subseteq \Delta_0$ satisfies
$\Gamma \cdot K_{\mm_{16}}=0$ and $\Gamma\cdot \Delta_0>0$.
\end{introtheorem}

\vskip 3pt

The curve $\Gamma$ is constructed by realizing a sufficiently general $2$-pointed curve $[C, x, y]\in \cM_{15,2}$ as a section on a $(2,2,2,2)$-complete intersection in $\PP^6$, which is induced by a linear system $L\in W^6_{19}(C)$, see also (\ref{cansurf}).  It turns out, this is a canonical surface $S\subseteq \PP^6$ on which the curve $C$ moves in a pencil, while keeping the marked points $x$ and $y$ fixed, that is, $\mbox{dim } \bigl|\I_{\{x,y\}/S}(C)\bigr|=1$. The curve $\Gamma\subseteq \mm_{16}$ is then induced by this pencil, see also (\ref{sweep2}) and its numerical characters are described in Proposition \ref{numbers}.

\vskip 4pt

\subsection{The non-abelian Brill-Noether divisor on $\mm_{16}$} The most important aspect of the paper lies in the choice of the second divisor. Via non-abelian Brill-Noether theory, we construct a uniruled divisor in $\mm_{16}$ swept by $K_{\mm_{16}}$-negative curves. For a smooth curve $C$ of genus $g$, we denote by $SU_C(2,\omega)$ the moduli space of equivalence classes of semistable rank $2$ vector bundles $E$ on $C$ with $\mbox{det}(E)\cong \omega_C$.  For $k\geq 0$, the non-abelian Brill-Noether variety
\[
SU_C(2,\omega, k):=\bigl\{E\in SU_C(2, \omega): h^0(C,E)\geq k\bigr\}
\]
is a Lagrangian degeneracy locus and each component of $SU_C(2,\omega, k)$ has dimension at least $\beta(2,g,k):=3g-3-{k+1\choose 2}$, see \cite{Mu}, \cite{BF}. Teixidor \cite[Theorem 1.1]{Te} showed that the Mukai-Petri map 
\begin{equation}\label{eq:mp}
\mu_E\colon \mbox{Sym}^2 H^0(C,E)\longrightarrow H^0(C, \mbox{Sym}^2 E)
\end{equation}
is injective for every vector bundle $E\in SU_C(2,\omega)$. In particular, each component of $SU_C(2,\omega, k)$ has dimension precisely $\beta(2,g,k)$. This result however does not establish the non-emptiness of $SU_C(2,\omega,k)$, when $\beta(2, g,k)\geq 0$.

\vskip 5pt

We are interested in the case
\[g=16 \ \mbox{ and } \  k=9,\]
when $\beta(2,16,9)=0$. First, by computing the fundamental class of $SU_C(2,\omega, 9)$ and then using a variation of Teixidor's work \cite{Te},  we show the following result:

\begin{introtheorem}\label{thm:existence}
A general curve $C$ of genus $16$ caries precisely $11$ stable vector bundle $E$ of rank $2$ with $\mathrm{det}(E)\cong \omega_C$ and $h^0(C, E)=9$.
Each of these bundles is globally generated and the corresponding scroll $\PP_E$ is embedded in $\PP H^0(C,E)^{\vee}\cong \PP^8$.
\end{introtheorem}

In particular, Theorem \ref{thm:existence} shows that $\mm_{16}$ possesses a modular cover of low degree. We let  $\mathcal{SU}_{16}(2, \omega, 9)$ be the moduli space classifying pairs $[C,E]$, where $[C]\in \cM_{16}$ and $E\in SU_C(2, \omega, 9)$ and consider the forgetful map of generic degree $11$
\begin{equation}\label{eq:varth}
\vartheta\colon \mathcal{SU}_{16}(2, \omega, 9)\longrightarrow \cM_{16}, \ \ \  [C, E]\mapsto [C].
\end{equation}

The low degree of $\vartheta$ is a  surprise, especially since, as explained in this paper,  the cardinality of $SU_C(2,\omega,k)$ for a general curve $[C]\in \cM_g$ when $\beta(2,g,k)=0$ grows 
very fast. It would be interesting to  geometrically construct these $11$ vector bundles for a general curve of genus $16$. 

\vskip 3pt

The first  instance where it has been pointed out that an analogue of the map $\vartheta$  can be used to derive important consequences on the geometry of $\mm_g$ was in \cite{FJP1}, where it was shown that the locus in $\mm_{13}$ consisting of curves $[C]$ such that $\mu_E$ is not an isomorphism for a vector bundle $E\in SU_C(2,\omega,8)$ leads to the first example of an effective divisor on $\mm_g$ having slope less that $6+\frac{10}{g}$. This quantity becomes relevant when trying to formulate replacements for the \emph{Slope Conjecture} \cite{FP}.

\vskip 5pt

In this paper we consider the locus $\mathfrak{D}_2$  which is the closure of the locus of those curves $[C]\in \cM_{16}$ carrying a stable rank $2$ vector bundle $F$ with 
$h^0(C,F)=7$ and $\mbox{det}(F)\cong \omega_C(-x-y-z)$,  for points $x,y,z\in C$. Using Theorem \ref{thm:existence}, it is not hard to see that $\mathfrak{D}_2$ is indeed expected to be a divisor in $\cM_{16}$. Starting with a stable vector bundle $E$ with $\mbox{det}(E)\cong \omega_C$ and $h^0(C,E)=9$, we perform an elementary transformation at three points $x, y,z\in C$ of the form 
\begin{equation}\label{eq:evxyz}
\mathrm{ev}\colon E\longrightarrow \ell_x\oplus \ell_y\oplus \ell_z,
\end{equation}
where $\ell_x\in \PP\bigl(E_{|x}^{\vee}\bigr)$, $\ell_y\in \PP\bigl(E_{|y}^{\vee}\bigr)$ and $\ell_z\in \PP\bigl(E_{|z}^{\vee}\bigr)$ are $1$-dimensional quotients in the fibres of $E$ at the respective points. Setting $F:=\mbox{Ker}(\mathrm{ev})$, we find $\mbox{det}(F)\cong \omega_C(-x-y-z)$ and the condition $h^0(C,F)\geq 7$ is equivalent to $\mbox{rk}(\mathrm{ev})\leq 2$. This is expected to be a codimension $7$ condition over the $6$-dimensional parameter space of tuples $(x,y,z, \ell_x, \ell_y,\ell_z)$ as above. By varying both $C$ and $E$ we expect to get in this way a divisor in $\cM_{16}$. An alternative description of the locus $\mathfrak{D}_2$ as the (push forward of the) universal trisecant divisor in $\mathcal{SU}_{16}(2, \omega)$ is provided by Proposition \ref{prop:d2=tris}. Finally, an explicit description of the tangent spaces of the fibres of the map $\mathfrak{D}_2\rightarrow \cM_{16}$ in terms of a twisted non-abelian Petri map is provided in Proposition \ref{prop:twisted_bn}.

\vskip 5pt

The key observation is that under generality assumptions, a curve $[C]\in \mathfrak{D}_2$ lies on a \emph{canonical surface} on which it moves in its linear system. Indeed, we may assume that such a vector bundle $F$ is globally generated and denote by 
$$\phi_F\colon C\longrightarrow G\bigl(2, H^0(C,F)^{\vee}\bigr)\cong \GG$$
the induced map and by $\mbox{det}\colon \bigwedge^2 H^0(C,F)\rightarrow H^0\bigl(C, \omega_C(-x-y-z)\bigr)$ the determinant map. Under the expectation that $\mathrm{det}$ is surjective, we have an induced injective map 
$$\mathrm{det}^{\vee}\colon \PP H^0\bigl(\omega_C(-x-y-z)\bigr)^{\vee} \longrightarrow  \PP\bigl(\bigwedge^2 H^0(C,F)^{\vee}\bigr)\cong \PP^{20}$$ and the following commutative diagram, where the vertical map $\mathfrak{p}$ is the Pl\"ucker embedding: 

$$\xymatrix@R=3pc@C=4pc{
C\ar[d]_{|\omega_C(-x-y-z)|}  \ar[r]^{\phi_F}& G\bigl(2, H^0(C,F)^{\vee}\bigr) \ar[d]_{\mathfrak{p}} \\
\PP H^0\bigl(\omega_C(-x-y-z)\bigr)^{\vee} \ar[r]^{\mathrm{det}^{\vee}} & \PP\bigl(\bigwedge^2 H^0(C,F)^{\vee}\bigr) \\
                 }
$$
In particular, $X:=\PP H^0\bigl(\omega_C(-x-y-z)\bigr)^{\vee}\cap \GG$ is a codimension $8$ linear section of the $10$-dimensional Grassmannian $\GG\cong G(2,7)$. By adjunction, since $K_{\GG}=\OO_{\GG}(-7)$, if the intersection is transverse, the  section $X$ is a canonical surface, that is, $\omega_X\cong \OO_X(1)$. The curve $C$  lies on $X$ and by adjunction $\OO_C(C)\cong \OO_C(x+y+z)$, therefore we have the following exact sequence
\begin{equation}\label{eq:seq1}
0\longrightarrow \OO_X\longrightarrow \OO_X(C)\longrightarrow \OO_C(x+y+z)\longrightarrow 0,
\end{equation}
from which we get $\mbox{dim } \bigl|\OO_X(C)\bigr|=1$. Denoting by $\Gamma_2\subseteq \mm_{16}$ the curve in moduli induced by this linear system, and recalling the notation $\lambda, \delta_0, \ldots, \delta_{\lfloor \frac{g}{2}\rfloor}$ for the standard generators of $CH^1(\mm_g)$, we compute that 
\begin{equation}\label{eq:int2}
\Gamma_2\cdot \lambda=\chi(X,\OO_X)+15=29 \ \ \mbox{ and } \ \ \Gamma_2\cdot \delta_0=189,
\end{equation}
The proof of (\ref{eq:int2}) is completed in Proposition \ref{prop:pencild2}. Here we used that $\mbox{deg}(X)=42$, which equals the degree of the Grassmannian $G(2,7)$. Furthermore, under the same generality assumptions, $\Gamma_2\cdot \delta_i=0$, for $i=1, \ldots, 8$. Since the canonical class of the moduli space is given by the well-known formula  $K_{\mm_{16}}\equiv 13\lambda-2\delta_0-3\delta_1-2\delta_2-\cdots-2\delta_8$, see also (\ref{canm}), we compute from (\ref{eq:int2})  that
 \begin{equation}\label{eq:int3}
\Gamma_2\cdot K_{\mm_{16}}=13\cdot 29-2\cdot 189=-1.
\end{equation}

We summarize the above considerations in the following:

\begin{introtheorem}\label{thm:bnr2}
The non-abelian Brill-Noether locus $\overline{\mathfrak{D}}_2$ possesses a divisorial component $\overline{\mathfrak{D}}_{\mathrm{rk}2}$ whose general point corresponds to a pair $[C,F]$ as above, such that the induced canonical surface $X\subseteq \GG$ is smooth. Furthermore, through a general point $[C]$ of $\overline{\mathfrak{D}}_{\mathrm{rk}2}$ there passes a rational curve $\Gamma_2\subseteq  \overline{\mathfrak{D}}_{\mathrm{rk}2}$ such that $\Gamma_2\cdot K_{\mm_{16}}=-1$. 
\end{introtheorem}

The proof of Theorem \ref{thm:bnr2} takes up a large portion of the paper and it involves a specialization to the locus of curves $[C]\in \cM_{16}$  for which the determinant map $\mathrm{det}\colon \bigwedge^2 H^0(C,F)\rightarrow H^0\bigl(C, \omega_C(-x-y-z)\bigr)$ is \emph{not} surjective, which means that the  Pl\"ucker image $P:=\langle C\rangle \subseteq \PP^{20}$  is $11$-dimensional. In this case, $C$ is linked to a curve $B$ of genus $4$ defined via the intersection
$$P\cap\GG=C+B.$$
The proof of Theorem \ref{thm:bnr2} consists of (i) constructing starting with $[B]\in \cM_4$ and an embedding $B\hookrightarrow \GG$ a point $[C,F]$ where all the conditions in Theorem \ref{thm:bnr2} are satisfied,  then (ii) deforming this construction to cover the case of a general point of $\dd2$.

\vskip 5pt

We now explain how Theorems \ref{sweep} and \ref{thm:bnr2} imply that $K_{\mm_{16}}$ in not pseudo-effective. Using \cite{BDPP}, this shows that $\mm_{16}$ is uniruled, thus proving Theorem \ref{main_thm}.  To that end, we fix an ample class $A\in \mathrm{Ample}(\mm_{16})$ and assume by contradiction $K_{\mm_{16}}$ is pseudo-effective, therefore for $\epsilon\in \mathbb Q_{>0}$ sufficiently small, we  have an effective $\mathbb Q$-representation 
\begin{equation}\label{eq:can_rep}
K_{\mm_{16}}+\epsilon\cdot A\equiv c_{\epsilon} \cdot\dd2 + D_{\epsilon},
\end{equation}
where $c_{\epsilon} \in \mathbb Q_{\geq 0}$ and $D_{\epsilon}$ is an effective $\mathbb Q$-divisors on $\mm_{16}$ not containing $\dd2$ in its support.  We intersect (\ref{eq:can_rep}) with the curve $\Gamma_2$, which we choose  general enough such that $\Gamma_2\nsubseteq D_{\epsilon}$ and $\Gamma_2\nsubseteq \Delta_0$. Since $\Gamma_2\cdot K_{\mm_{16}}=-1$,  we obtain $$-1+\epsilon\ \bigl(\Gamma_2 \cdot A\bigr)=c_{\epsilon}\ \bigl(\Gamma_2 \cdot \dd2\bigr)+
\bigl(\Gamma_2\cdot D_{\epsilon}\bigr) \geq c_{\epsilon}\ \bigl(\Gamma_2 \cdot \dd2\bigr),$$ implying $\Gamma_2\cdot \dd2<0$. Furthermore, $$c_{\epsilon}\geq \frac{1-\epsilon\ \bigl(\Gamma_2\cdot A)}{-\bigl(\Gamma_2\cdot \dd2\bigr)},$$ which shows that the coefficients $c_{\epsilon}$ do not converge to $0$. Next we intersect
(\ref{eq:can_rep}) with the sweeping curve $\Gamma_1$ from Theorem \ref{sweep}. Since $\Gamma_1\cdot K_{\mm_{16}}=0$, we obtain 
$$-\epsilon\ \bigl(\Gamma_1\cdot A\bigr)=c_{\epsilon}\ \bigl(\Gamma_1 \cdot \dd2) + \bigl(\Gamma_1 \cdot D_{\epsilon}\bigr)\geq c_{\epsilon}\ \bigl(\Gamma_1\cdot \dd2\bigr),$$ from which, also using that $\Gamma_1\cdot \Delta_0>0$, it follows that
$$0\leq \Gamma_1\cdot \dd2 \leq -\bigl(\Gamma_2\cdot \dd2\bigr)\cdot \bigl(\Gamma_1\cdot A\bigr)\cdot \frac{\epsilon}{1-\epsilon\cdot \bigl(\Gamma_2\cdot A\bigr)}.$$
Passing to the limit, we obtain $\Gamma_1\cdot \dd2$=0, that is, $\dd2$ has slope $\frac{13}{2}$. We then finish the proof by showing in Theorem \ref{thm:slope-ruleout} that no divisorial component of the non-abelian Brill-Noether locus $\overline{\mathfrak{D}}_2$ can possibly have slope $\frac{13}{2}$, in particular $K_{\mm_{16}}$ is not pseudo-effective. Since it is also proved in \cite{HM} that if $\widehat{\cM}_g\rightarrow \mm_g$ is any resolution of singularities of $\mm_g$, for any $\ell\geq 0$ one has an isomorphism $H^0\bigl(\mm_g, K_{\mm_g}^{\otimes \ell}\bigr)\stackrel{\cong}\longrightarrow H^0\bigl(\widehat{\cM}_g, K_{\widehat{\cM}_g}^{\otimes \ell}\bigr)$,  this finishes the proof of Theorem \ref{main_thm} and shows that $\mm_{16}$ is uniruled.

\vskip 6pt

\noindent {\small {\bf Acknowledgments:} We are most grateful to J\'anos Koll\'ar for a very useful observation concerning Theorem \ref{main_thm}, pointing out that our proof implies the non \emph{pseudo}-effectiveness of the canonical class $K_{\mm_{16}}$, rather than merely saying that the Kodaira dimension of $\mm_{16}$ is negative. 

\vskip 9pt
 
Farkas was  supported by the Berlin Mathematics Research Center MATH+, by the Einstein Foundation Berlin and by the ERC Advanced Grant SYZYGY (grant No. 834172). This work was started when Verra was an invited professor of the Berlin Mathematical School.

The work in this paper has been completed by late 2024. The first named author lectured on these results at several major conferences in algebraic geometry including Taipei 2024, San Diego 2024, Frankfurt 2025 and Hannover 2026.}

\section{The uniruled parametrization of $\mm_{15}$}
The parametrization of the boundary divisor $\Delta_0$ of $\mm_{16}$ and the proof of Theorem \ref{sweep} uses several results from \cite{BV}, which we now recall. We denote by $\H_{15,9}$ the Hurwitz space parametrizing degree $9$ covers $C\rightarrow \PP^1$ having simple ramification, where $C$ is a smooth curve of genus $15$. Then $\H_{15,9}$ is birational to the parameter space $\G_{15,9}^1$ classifying pairs $(C, A)$, where $[C]\in \cM_{15}$ and $A\in W^1_9(C)$ is a pencil. By residuation, $\mathcal{G}^1_{15,9}$ is isomorphic to the parameter space $\mathcal{G}^6_{15,19}$ of pairs $[C,L]$, where $C$ is a smooth curve of genus $15$ and $L\in W^6_{19}(C)$. Note that the general fibre of the forgetful map
\begin{equation}\label{eq:pi}
\pi \colon \mathcal{H}_{15, 9}\longrightarrow \cM_{15}, \ \ \ \ [C, A]\mapsto [C]
\end{equation} is $1$-dimensional. Clearly, $\H_{15,9}$ and thus $\mathcal{G}_{15,19}^6$ is irreducible.

\vskip 4pt

We pick a general element $[C,L]\in \G^6_{15, 19}$, in particular $L$ is very ample and $h^0(C,L)=7$. We set $A:=\omega_C\otimes L^{\vee}\in W^1_9(C)$. We may assume that $A$ is base point free and the pencil $|A|$ has simple ramification. We consider the multiplication map
$$\phi_L\colon \mbox{Sym}^2 H^0(C,L)\longrightarrow H^0(C,L^2).$$
Since $C$ is Petri general, $h^1(C,L^2)=0$, therefore $h^0(C,L^2)=2\cdot 19+1-15=24$. Furthermore, via a degeneration argument it is shown in \cite[Theorem 3.11]{BV}, that for a general choice of $(C,L)$, the map $\phi_L$ is surjective, hence $h^0(\PP^6, \mathcal{I}_{C/\PP^6}(2))=\mbox{dim}\bigl(\mbox{Ker}(\phi_L)\bigr)=4$, that is, the degree $19$ curve $C\subseteq \PP^6$ lies on precisely $4$ independent quadrics. We let
\begin{equation}\label{cansurf}
S:=\mathrm{Bs}\bigl|\mathcal{I}_{C/\PP^6}(2)\bigr|
\end{equation}
be the base locus of the system of quadrics containing $C$. It is further established in \cite[Theorem 3.11]{BV} that under our generality assumptions, $S$ is a smooth surface. From the adjunction formula, since $\omega_{\PP^6}\cong \OO_{\PP^6}(-7)$, it follows that $\omega_S=\OO_S(1)$, that is, $S$ is a canonical surface. We write down the exact sequence
\begin{equation}\label{ex_seq}
0\longrightarrow \OO_S\longrightarrow \OO_S(C)\longrightarrow \OO_C(C)\longrightarrow 0.
\end{equation}

From the adjunction formula $\OO_C(C)\cong \omega_C\otimes \omega_{S|C}^{\vee}=\omega_C\otimes L^{\vee}=A\in W^1_9(C)$. Since $S$ is a regular surface, by taking cohomology in (\ref{ex_seq}), we obtain $$h^0(S, \OO_S(C))=h^0(S,\OO_S)+h^0(C,A)=3.$$ Observe also from the sequence (\ref{ex_seq}) that the linear system $\bigl|\OO_S(C)\bigr|$ is base point free, because $\bigl|\OO_C(C)\bigr|=|A|$ is so. This brings to an end our summary of the results from \cite{BV}.

\vskip 5pt

In what follows, we denote by \begin{equation}\label{finitemap}
f\colon S\longrightarrow \PP^2=\bigl|\OO_S(C)\bigr|^{\vee}
\end{equation}
the induced map. For what we intend to do, it is important to show that $f$ is a finite map, or equivalently, that $\OO_S(C)$ is ample.

\begin{theorem}\label{prop:ample}
For a general pair $(C, A)\in \mathcal{H}_{15,9}$, the line bundle $\OO_S(C)$ is ample.
\end{theorem}

In order to prove Theorem \ref{prop:ample} it suffices to exhibit a single pair $[C, A] \in \H_{15,9}$ for which the corresponding map
$f\colon S\rightarrow \PP^2$ given by (\ref{finitemap}) is finite. We shall realize the canonical surface $S\subseteq \PP^6$ as the double cover of a  $K3$ surface $Y\subseteq \PP^5$ of genus $5$ (that is, of degree $8$) having a certain Picard lattice of rank $3$. 

\begin{definition}\label{lattice3} We denote by $\Lambda$ the even lattice of signature $(1,2)$ generated by elements $\bar{H}, \bar{F}$ and $\bar{R}$ having the following intersection matrix:
$$ \begin{pmatrix}
    \bar{H}^2 & \bar{H}\cdot \bar{F} & \bar{H}\cdot \bar{R}\\
     \bar{F}\cdot \bar{H} & \bar{F}^2 & \bar{F}\cdot \bar{R} \\
   \bar{R}\cdot \bar{H} & \bar{R}\cdot \bar{F} & \bar{R}^2\\
   \end{pmatrix}
   =
   \begin{pmatrix}
 8 & 9& 1\\
9 & 4 & 2 \\
1 & 2 & -2\\
\end{pmatrix}.
$$

We denote by $\F_5^{\Lambda}$ the moduli space of polarized $K3$ surfaces $[Y,\bar{H}]$, where $\bar{H}^2=8$, admitting a primitive embedding
$\Lambda\hookrightarrow \mbox{Pic}(Y)$, such that the classes $\bar{H}, \bar{F}, \bar{R}$ correspond to curve classes on $Y$ which we denote by the same symbol. Furthermore, $\bar{H}\in \mathrm{Pic}(Y)$ is assumed to be ample.
\end{definition}

For details on the construction of the moduli space $\F_5^{\Lambda}$ we refer to \cite[Section 3]{Do}. It follows from \emph{loc.cit.} that $\F_5^{\Lambda}$ is an irreducible variety of dimension $17=20-\mathrm{rk}(\Lambda)$. Let us now fix a general element $[Y, \bar{H}]$, where
$\mbox{Pic}(Y)\cong \mathbb Z\langle \bar{H}, \bar{F}, \bar{R}\rangle$ as in Definition \ref{lattice3}. Then $\OO_Y(\bar{H})$ is very ample and we denote by
\begin{equation}\label{eq:Y}
\varphi_{\bar{H}}\colon Y\hookrightarrow \PP^5
\end{equation}
the embedding induced by this linear system. Observe that $h^0\bigl(\PP^5, \mathcal{I}_{Y/\PP^5}(2)\bigr)=3$ and that $Y=\mathrm{Bs}\bigl|\mathcal{I}_{Y/\PP^5}(2)\bigr|$ is in fact a complete intersection of three quadrics. Note that $\bar{F}\subseteq Y$ is a curve of genus $3$, whereas $\bar{R}\subseteq Y$ is a smooth rational curve embedded as a line under the map $\varphi_{\bar{H}}$. The class $\bar{E}:=2\bar{H}-\bar{F}$ satisfies $\bar{E}^2=0$. Since $\bar{E}\cdot \bar{H}=7>0$, it follows that  $\bigl|\bar{E}\bigr|$ is an elliptic pencil and furthermore $\bar{E}\cdot \bar{R}=0$. Setting also
$$\bar{D}:=2\bar{H}+\bar{R}-\bar{E}\in \mathrm{Pic}(Y),$$
we compute $\bar{D}^2=6$, $\bar{D}\cdot \bar{E}=14$ and $\bar{D}\cdot \bar{R}=0$. In the basis $\bigl(\bar{D}, \bar{E}, \bar{R}\bigr)$ of $\mathrm{Pic}(Y)$, the intersection form on $Y$ is described by the following simpler matrix:
\begin{equation}\label{eq:intmatrix}
\begin{pmatrix}
    \bar{D}^2 & \bar{D}\cdot \bar{E} & \bar{D}\cdot \bar{R}\\
     \bar{E}\cdot \bar{D} & \bar{E}^2 & \bar{E}\cdot \bar{R} \\
   \bar{R}\cdot \bar{D} & \bar{R}\cdot \bar{E} & \bar{R}^2\\
   \end{pmatrix}
   =
   \begin{pmatrix}
 6 & 14& 0\\
14 & 0 & 0 \\
0 & 0 & -2\\
\end{pmatrix}.
\end{equation}

On our way to proving Theorem \ref{prop:ample}, we  establish the following result:

\begin{proposition}\label{fample}
The line bundle $\OO_Y(\bar{F})$ is very ample.
\end{proposition}
\begin{proof}
We first claim that $\bar{F}$ is nef. Since $\bar{F}^2=4>0$, it suffices to check that for any smooth rational curve
$\bar{\Gamma}\subseteq Y$, one has $\bar{\Gamma} \cdot \bar{F}\geq 0$. We write $\bar{\Gamma} \equiv a\bar{D}+b\bar{E}+c\bar{R}$, where $a,b$ and $c$ are integers.
We may assume $\bar{\Gamma} \neq \bar{R}$, thus $\bar{\Gamma} \cdot \bar{R}\geq 0$, implying $c\leq 0$. Furthermore, $\bar{\Gamma} \cdot \bar{E}\geq 0$, hence $a\geq 0$. Using (\ref{eq:intmatrix}), one has $\bar{\Gamma}^2=6a^2-2c^2+28ab=-2$. Assume by contradiction $\bar{\Gamma}\cdot \bar F=\bar{\Gamma}\cdot (\bar{D}-\bar{R})=6a+14b+2c\leq -2$. Multiplying this inequality with $2a\geq 0$ and substituting in the equality $\bar{\Gamma}^2=-2$  we obtain that $(a+c)^2+2a^2+2a-1\geq 0$, implying $a=0$ and $c\in \{-1,1\}$. If, say $c=1$, then $\bar{\Gamma} \equiv \bar{R}+b\bar{E}$. From the assumption $\bar{\Gamma}\cdot \bar{F}\leq -2$, we obtain that $b\leq -1$, hence $\bar{\Gamma}\cdot \bar{H}<0$, thus $\bar{\Gamma}$ cannot be effective, a contradiction. The case $c=-1$, implying $b\leq 0$ is ruled out similarly

\vskip 4pt

Thus $\bar{F}$ is a nef curve. To conclude that $\bar{F}$ is very ample, we invoke \cite{SD}. It suffices to rule out the existence of a divisor class $\bar{M}\in \mathrm{Pic}(Y)$ such that (i) $\bar{M}^2=0$ and $\bar{M}\cdot \bar{F}\in \{1,2\}$, or satisfying (ii) $\bar{M}^2=-2$ and $\bar{M}\cdot \bar{F}=0$. We discuss only (i), the remaining case being similar. Write $\bar{M}=a\bar{D}+b\bar{E}+c\bar{R}$. Since $\bar{M}^2=0$, from (\ref{eq:intmatrix}) we obtain $3a^2-c^2+14ab=0$, whereas from $\bar{M}\cdot \bar{F}=2$, we obtain that $3a+7b+c=1$. Eliminating $c$, we find $6a^2+a(28b-6)+49b^2-14b+1=0$. Since the discriminant of this equation is negative, this case is excluded. We conclude that $\bar{F}$ is very ample.
\end{proof}

\vskip 4pt

We fix a general polarized $K3$ surface $[Y,\bar{H}]\in \mathcal{F}_5^{\Lambda}$, while keeping the notation from above. Choose a smooth divisor $\bar{Q}\in \bigl|\OO_Y(2\bar{H})\bigr|$ and consider the double cover
\begin{equation}\label{eq:S}
\sigma\colon S\longrightarrow Y
\end{equation}
branched along $\bar{Q}$. We denote by $Q\subseteq S$ the ramification divisor of $\sigma$, hence $\sigma^*(\bar{Q})=2Q$. We set $H:=\sigma^*(\bar{H})$, where $\bar{H}\in \bigl|\OO_Y(1)\bigr|$ is a linear section of $Y$. Note that $Q\in \bigl|\OO_S(H)\bigr|$.

\begin{proposition}
The induced morphism $\varphi_H\colon S\rightarrow \PP^6$ embeds $S$ as a canonical surface which is the complete intersection of $4$ quadrics in $\PP^6$. More precisely, $S$ is a quadratic section of the cone $\mathcal{C}_Y\subseteq \PP^6$ over the $K3$ surface $Y\subseteq \PP^5$.
\end{proposition}

\begin{proof}
From the adjunction formula we find $\omega_S=\OO_S(Q)=\OO_S(H)$. Furthermore, we have $\sigma_*(\OO_S)=\OO_Y\oplus \OO_Y(-H)$, hence from the projection formula we can write
$$H^0\bigl(S,\OO_S(H)\bigr)\cong H^0\bigl(Y, \OO_Y(\bar{H})\bigr)\oplus H^0\bigl(Y,\OO_Y\bigr)\cong H^0\bigl(Y, \OO_Y(\bar{H})\bigr)\oplus \mathbb C\langle Q\rangle,$$
where recall that $Q\in \bigl|\OO_S(H)\bigr|$, as well as
$$H^0\bigl(S,\OO_S(2H)\bigr)\cong H^0\bigl(Y, \OO_Y(2\bar{H})\bigr)\oplus H^0\bigl(Y,\OO_Y(\bar{H})\bigr)\cdot Q.$$

Thus $h^0\bigl(S, \OO_S(H)\bigr)=6$ and $h^0(S,\OO_S(2H))=h^0(Y,\OO_Y(2))+h^0(Y,\OO_Y(1))=2+2\bar{H}^2+6=24$. Furthermore,  $S\subseteq \PP^6$ is projectively normal, so
$h^0\bigl(\PP^6, \I_{S/\PP^6}(2)\bigr)=4$. Since clearly $S\subseteq \cC_Y$, it follows that $S$ can be viewed as a quadratic section of the cone $\cC_Y$, precisely the intersection of $\cC_Y$ with one of the quadrics containing $S$ not lying in the subsystem $\bigl|\sigma^* H^0\bigl(\PP^5, \I_{Y/\PP^5}(2)\bigr)\bigr|$.
\end{proof}

%\vskip 3pt

We are now in a position to prove Theorem \ref{prop:ample}. We denote by $\mathrm{Hilb}_{15,19}$ the unique component of the Hilbert scheme of curves $C\subseteq \PP^6$ of genus $15$ and degree $19$ dominating $\cM_{15}$. A general point of $\mathrm{Hilb}_{15,19}$ corresponds to a smooth projectively normal curve $C\subseteq \PP^6$ such that the canonical surface $S$ defined by (\ref{cansurf}) is smooth.

\vskip 4pt

\noindent \emph{Proof of Theorem \ref{prop:ample}.} We choose a $K3$ surface $[Y, \OO_Y(\bar{H}]\in \F_5^{\Lambda}$ with $\mbox{Pic}(Y)=\mathbb Z\langle \bar{H}, \bar{F}, \bar{R}\rangle$, where the intersection matrix is given as in Definition \ref{lattice3}. The restriction map
$$H^0\bigl(Y, \OO_Y(2\bar{H})\bigr)\longrightarrow H^0\bigl(\bar{R}, \OO_{\bar{R}}(2\bar{H})\bigr)$$  being surjective, we can choose a smooth curve
$\bar{Q}\in \bigl|\OO_Y(2\bar{H})\bigr|$ which is \emph{tangent} to $\bar{R}$, that is, $\bar{Q}\cdot \bar{R}=2y$, for a point $y\in Y$. Construct the double cover $\sigma\colon S\rightarrow Y$ defined in (\ref{eq:S}).
The pull-back $\sigma^*(\bar{R})$ is then a double cover of $\bar{R}$ branched over the single point $y$, hence necessarily
$$\sigma^*(\bar{R})=R+R'\subseteq S,$$
where $R$ and $R'$ are lines on $S\subseteq \PP^6$ meeting at a single point. Next, we choose a smooth genus $3$ curve $\bar{F}\subseteq Y$ general in its linear system and set
$$C':=\sigma^*(\bar{F})\subseteq S.$$
Since $\bar{F}\cdot \bar{Q}=2\bar{F}\cdot \bar{H}=18$, we obtain that $C'$ is a smooth curve of genus $14$ and degree $18$ endowed with the double cover $C'\rightarrow \bar{F}$. Note that the linear system $$\bigl|\OO_S(C')\bigr|=\pi^*\bigl|\OO_Y(\bar{F})\bigr|$$ is $3$-dimensional. Applying Theorem \ref{fample}, since $\OO_Y(\bar{F})$ is ample and $\sigma$ is finite, we obtain that $\OO_S(C')$ is ample as well. Observe that $C'\cdot R=\bar{F}\cdot \bar{R}=2$. Choosing $\bar{F}$ general in its linear system, we can arrange the intersection of $R$ and $C'$ to be transverse, therefore
\begin{equation}\label{eq:Cdeg}
C:=C'+R\subseteq S\subseteq \PP^6
\end{equation}
is a nodal curve of genus $15$ and degree $19$. Note that the linear system $\bigl|\OO_S(C)\bigr|$ has $R$ as a fixed component, and
$\bigl|\OO_S(C)\bigr|=R+\pi^*\bigl|\OO_Y(\bar{F})\bigr|$.

\vskip 4pt

Despite the fact that $\bigl|\OO_S(C)\bigr|$ is not ample, we can complete the proof of Theorem \ref{prop:ample}. Indeed, let us pick a general family $\bigl\{[C_t\hookrightarrow \PP^6]\bigr\}_{t\in T}\subseteq  \mathrm{Hilb}_{15,19}$ over a pointed base $( T, o)$, whose fibre over $o\in T$ is the curve $C$ described in (\ref{eq:Cdeg}). If $S_t=\mathrm{Bs} \bigl|\I_{C_t/\PP^6}(2)\bigr|$, assume the line bundle $\OO_{S_t}(C_t)$ is not ample for each $t\in T$. As we have already observed, we may assume that $\OO_{S_t}(C_t)$ is nef for all $t\in T$ and we denote by $f_t\colon S_t\rightarrow \PP^2$ the map induced by the linear system $\bigl|\OO_{S_t}(C_t)\bigr|$ for $t\in T\setminus \{o\}$. The limiting map of this family
$$f_o\colon S\longrightarrow \PP^2,$$ satisfies $f_o^*\bigl(\OO_{\PP^2}(1)\bigr)=\OO_S(R+C')$ and is induced by a subspace of sections $\sigma^*(V)$, where  $V\subseteq H^0\bigl(Y, \OO_Y(\bar{F})\bigr)$ is $3$-dimensional. By assumption, there exists a family of curves $\Gamma_t\subseteq S_t$ such that $\Gamma_t\cdot C_t=0$.  We denote by $\Gamma_o\subseteq S$ the limiting curve of $\Gamma_t$, therefore $\Gamma_o\cdot (C'+R)=0$.
We write $\Gamma_o=G+mR$, where $m\geq 0$ and $G\subseteq S$ is a curve not having $R$ in its support. From the adjunction formula, we find $R^2=-3$. Since $R\cdot C'=2$, it follows that $R\cdot (C'+R)=1$,  thus $G\neq 0$. Furthermore, the morphism $f_o$ contracts $G$, which we argue,  leads to a contradiction. Indeed, $f_o$ admits a factorization
$$\xymatrix{
    S \ar[r]^{\sigma}  \ar@/^2pc/[rrr]^{f_o} &   Y  \ar[r]^{|\bar{F}|}  &  \PP^3  \ar[r]^p & \PP^2 \\
   }
$$
where $p\colon \PP^3\rightarrow \PP^2$ is the linear projection corresponding to $V\subseteq H^0\bigl(Y, \OO_Y(\bar{F})\bigr)$. Since $\sigma$ is finite and $|\bar{F}|$ is very ample, it follows that $\sigma(G)$ must be contracted by the projection $p$, that is, $\sigma(C)$ is a line in $\PP^3$. By inspecting the intersection matrix (\ref{eq:intmatrix}) of $\mbox{Pic}(Y)$ we immediately see that no such line can exist on $Y$, which finishes the proof.
\hfill $\Box$

\section{The uniruledness of the boundary divisor $\Delta_0$ in $\mm_{16}$}

We now lift the construction discussed above from $\mm_{15}$ to the moduli space $\mm_{15,2}$ of $2$-pointed stable curves of genus $15$ and eventually to $\mm_{16}$. Recall that  $\mathrm{Hilb}_{15,19}$ is the component of the Hilbert scheme of curves $C\subseteq \PP^6$ of genus $15$ and degree $19$ dominating $\cM_{15}$.  We denote by $\mathrm{Hilb}_{2,2,2,2}$ the Hilbert scheme of complete intersections of $4$ quadrics in $\PP^6$. Since $\mathrm{Hilb}_{15,19}\dblq PGL(7)$ is birational to the Hurwitz space $\mathcal{H}_{15,9}$, we have a rational map
\begin{equation}\label{eq:P2}
\chi \colon \mathcal{H}_{15,9}\dashrightarrow \mathrm{Hilb}_{2,2,2,2}\dblq PGL(7), \ \ \ \ [C,A]\mapsto S:=\mathrm{Bs}\bigl|\mathcal{I}_{C/\PP^6}(2)\bigr| \ \ \ \mbox{mod } PGL(7),
\end{equation}
where the canonical surface $S\subseteq \PP^6$ is defined by (\ref{cansurf}).  We set
\begin{equation}\label{def:S}
\mathcal{S}:=\chi(\mathcal{H}_{15,9}).
\end{equation}

The general fibre of the morphism $\chi\colon \H_{15,9}\rightarrow \cS$ consists of finitely many linear nonempty open subsets of linear systems
$\bigl|\OO_S(C)\bigr|$, where $C\subseteq S\subseteq \PP^6$ is a smooth curve of genus $15$ and degree $19$. In particular, $\cS$ is an irreducible variety of dimension $41=\mathrm{dim}(\H_{15,9})-2$. Recall that $\pi\colon \H_{15,9}\rightarrow \cM_{15}$ denotes the forgetful map. The next observation will prove to be useful in several moduli counts.

\begin{proposition}\label{dims}
If $\cS'$ is an irreducible subvariety of $\cS$ of dimension $\mathrm{dim}(\cS')\leq 39$, then $\pi\bigl(\chi^{-1}(\cS')\bigr)$ is a proper subvariety of $\cM_{15}$.
\end{proposition}
\begin{proof}
Since $\mathrm{dim}\bigl(\chi^{-1}(\cS')\bigr)\leq \mathrm{dim}(\cS')+2\leq 41=\mathrm{dim}(\cM_{15})-1$, the claim follows.
\end{proof}

\vskip 4pt

Let us now take a general curve $C$ of genus $15$ and consider the correspondence
$$\Sigma:=\Bigl\{(A, x+y)\in W^1_9(C)\times C_2: H^0(C, A(-x-y))\neq 0\Bigr\},$$
endowed with the projections $\pi_1\colon \Sigma\rightarrow W^1_9(C)$ and $\pi_2\colon \Sigma\rightarrow C_2$ respectively. Here $C_2$ is the second symmetric product of $C$.  It follows that $\Sigma$ is an irreducible surface and that $\pi_2$ is generically finite. Indeed, for a general point $2x\in C_2$, we can invoke for instance
\cite[Theorem 1.1]{EH} to conclude that $\pi_2^{-1}(2x)$ is finite. The fibre $\pi_1^{-1}(A)$ is irreducible whenever $A$ has simple ramification.

\vskip 4pt

We now fix a general element $[C, x,y]\in \mm_{15,2}$. Then there exist finitely many pencils $A\in W^1_9(C)$ containing both points $x$ and $y$ in the same fibre. Each of these pencils $A$ may be assumed to be base point free with simple ramification and general enough such that $L:=\omega_C\otimes A^{\vee}\in W^6_{19}(C)$ is very ample and in the embedding
$$\varphi_L\colon C\hookrightarrow \PP^6$$
the curve $C$ lies on precisely $4$ independent quadrics intersecting in a smooth canonical surface $S$ defined by (\ref{cansurf}).

\begin{proposition}\label{penc2} 
With the notation above, if $h^0\bigl(C, A(-x-y)\bigr)=1$, then $\mathrm{dim} \ \bigl|\mathcal{I}_{\{x,y\}}(C)\bigr|=1$.
\end{proposition}
\begin{proof} It  follows from the commutativity of the following diagram, keeping in mind that $h^0(S, \OO_S(C))=3$ and that the first column is injective.
$$
 \xymatrix{
         0 \ar[r] & H^0\bigl(S, \mathcal{I}_{\{x,y\}}(C)\bigr) \ar[r]^{} \ar[d]_{} & H^0\bigl(S, \OO_S(C)\bigr) \ar[d]_{\mathrm{res}} \ar[r]^{} & H^0\bigl(\OO_{\{x,y\}}(C)\bigr) \ar[d]_{=}  & \\
          0 \ar[r] & H^0\bigl(C,A(-x-y)\bigr) \ar[r] & H^0(C, A) \ar[r]      & H^0\bigl(\OO_{\{x,y\}}(C)\bigr)   & }
$$
\end{proof}

We now introduce the moduli map of the pencil introduced in Proposition \ref{penc2}
\begin{equation}\label{penc3}
m\colon \PP=\bigl|\mathcal{I}_{\{x,y\}}(C)\bigr|\longrightarrow \mm_{15,2},
\end{equation}
where the marked points of the pencil are the base points $x$ and $y$ respectively.
Composing $m$ with the clutching map
$\mm_{15,2}\twoheadrightarrow \Delta_0\subseteq \mm_{16}$, we obtain a pencil $\xi\colon \PP \rightarrow \Delta_0$. We set
\begin{equation}\label{sweep2}
R:=m_*(\PP)\subseteq \mm_{15,2}   \ \ \mbox{ and  } \ \ \Gamma:=\xi_*(\PP)\subseteq \mm_{16}.
\end{equation}

\vskip 4pt

\begin{proposition}\label{prop:irr}
Every curve inside the pencil $\Gamma \subseteq \mm_{16}$ corresponds to a nodal curve which does not belong to any of the boundary divisors $\Delta_1, \ldots, \Delta_8$ .
\end{proposition}

\begin{proof}
Keeping the notation above, for a generic choice of $(A, x+y)\in \Sigma$, the pencil $$\PP:=\bigl|\mathcal{I}_{\{x,y\}}(C)\bigr|$$ corresponds to a generic line inside $\bigl|\OO_S(C))\bigr|$. As pointed out in Theorem \ref{prop:ample}, $\bigl|\OO_S(C)\bigr|$ is base point free and ample on the surface $S$ defined by (\ref{cansurf}), giving rise to the finite map
$$
f\colon S\longrightarrow \PP^2=\bigl|\OO_S(C)\bigr|^{\vee}
$$
considered in (\ref{finitemap}). We show that the inverse image $\PP$ under $f$ of a general pencil of lines in $\bigl|\OO_S(C)\bigr|^{\vee}$ consists only of integral curves with at most one node. This is achieved in several steps.

\vskip 5pt

\noindent {\bf(i)} Since $\OO_S(C)$ is ample, we can apply \cite[Theorem A]{BL} and conclude that each curve $C'\in \bigl|\OO_S(C)\bigr|$ is $2$-connected, that is, it cannot be written as a sum of effective divisors $C'=F+M$, where $F\cdot M\leq 1$. This implies that $\bigl|\OO_S(C)\bigr|$ does not contain any \emph{tree-like} curves, that is, curves for which its irreducible components meet at a single point, which furthermore is a node.

%We rule out the possibility that $\PP$ contains a reducible $1$-nodal curve $C'=F+M$, where $F$ and $M$ are integral smooth curves on $S$ such that $F\cdot M = 1$. Assume this is the case. By the Hodge Index Theorem,  $F^2\cdot M^2<1$. Assume $F^2 = 0$. For a general element $C'\in \PP$, we have $F\cdot C' = 1$. Moreover a general $C'$ intersects  $F$ at one point
%$q$ such that $\mathcal O_F(C') \cong \mathcal O_F(q)$. If $h^0(F, \mathcal O_F(q)) = 1$, then it is easily seen that $q$ is a base point of $\bigl|\OO_S(C)\bigr|$, therefore $q$ is a base point
%of $|A|=\bigl|\OO_C(C)\bigr|$. This contradicts the generality of  $C$, which has gonality $9$. If, on the other hand, $h^0(F, \mathcal O_F(q)) \geq 2$, then $F$ is a smooth rational curve. Since $F^2 = 0$, it follows $F\cdot K_S = -2$, contradicting  the very ampleness of
%$K_S$. The same argument works if $M^2 = 0$, hence we may assume $F^2\cdot M^2 \neq 0$.

%\vskip 3pt

 %Let $F^2 \leq -2$ then $F\cdot C' \leq -1$ and $F$ is a
%fixed component of $\bigl|\OO_S(C)\bigr|$, a contradiction again. Finally, assume $F^2 = -1$ and then $F\cdot C' = 0$. Choose a general point $z\in F\setminus M$. Since $\bigl|\OO_S(C)\bigr|$ is base point free, it follows that $\bigl|\mathcal{I}_{\{z\}}(C)\bigr|$ is a pencil. Each curve of this pencil must contain $F$, thus  that $|M|=|C-F|$ is a
%pencil and  $h^0(F, \mathcal O_F(M))=2$. Hence $F$ is an exceptional line on the minimal surface $S$, a contradiction. Summarizing, $\PP$ cannot contain any curves of compact type.

\vskip 6pt

\noindent {\bf{(ii)}}  The essential step in our argument involves proving that $\PP$ contains no curves with singularities worse than nodes. Precisely, we show that $\bigl|\OO_S(C)\bigr|$ contains only \emph{finitely many} non-nodal curves.  Note first that the \emph{branch curve} $B\subseteq \PP^2$ of $f$ is reduced, else we contradict the assumption that the pencil $A\in W^1_9(C)$ on $C$ has simple ramification.  We introduce the discriminant curve $$J:=\Bigl\{C'\in \bigl|\OO_S(C)\bigr|: C' \mbox{ is singular}\Bigr\}.$$

The dual curve $B^{\vee}$ is contained in $J$. Since $B$ is reduced, the general tangent line to $B$ is tangent at exactly one point $p\in B$ and with multiplicity $2$. A standard local calculation shows that $f^*\bigl(\mathbb T_p(B)\bigr)\in |\OO_S(C)|$ is a one-nodal curve, singular at exactly one point $z \in f^{-1}(p)$ such that the differential $df_z\colon T_z(S)\rightarrow T_p(\PP^2)$ is not an isomorphism.

\vskip 4pt

The complement
$J\setminus B^{\vee}$ is the (possibly empty) union of (some of) the pencils $\mathbf P_b$, where  $b \in B_{\mathrm{sing}}$ and $\PP_b$ is defined as
the pull-back by $f$ of the pencil of lines in $\PP^2$ through $b$. In view of the numerical situation at hand (that is, $C^2=9$), the geometric possibilities for such a pencil $$\PP_b \subseteq J$$ are quite constrained.  Since $f$ is finite, the pencil $\mathbf P_b$ has no fixed components. Let $Z:=\mathrm{Bs}(\PP_b)$. Then a general $C' \in \PP_b$ is integral and smooth along $C'\setminus Z$. Moreover, each $C' \in \PP_b$ is singular at a given point $z \in Z$ and a general such $C'$  has multiplicity $m\geq 2$ at $z$. Necessarily, the differential $df_z\colon T_z(S)\rightarrow T_p(\PP^2)$ is zero. Since $m^2 \leq C^2=(C')^2 = 9$, we find $m\in \{2, 3\}$.  Let
$$
\sigma\colon S' \longrightarrow S
$$
be the blow-up of $S$ at $z$ and denote by $E\subseteq S'$ the exceptional divisor. The pencil
 $\bigl|\mathcal O_{S'}(\sigma^*C -  mE)\bigr|$
is the strict transform of $\PP_b$. Observe that the restriction map
$$
r\colon H^0\bigl(S', \mathcal O_{S'}(\sigma^*C - mE)\bigr) \longrightarrow H^0\bigl(E, \mathcal O_{E}(m)\bigr)
$$
is not zero, hence $\mathrm{Im}(r)$ defines a linear series $\mathfrak{p}_b$ on $E\cong \PP^1$. Either $\mathfrak{p}_b$ is a pencil or a constant divisor of degree $m\in \{2, 3\}$.  We now list the possibilities for the pencil $\PP_b$.

\vskip 4pt

\noindent {\bf{(P1)}} If $m = 3$, then $\mbox{supp}(Z)=\{z\}$. Every curve $C'\in \PP_b$ has a triple point  at $z$.

\vskip 3pt

\noindent {\bf{(P2)}} If $m = 2$, then either each $C'\in \PP_b$ has a node, or else, each $C'\in \PP_b$ has a cusp at $z$. Indeed, if $\mathfrak{p}_b$ is a pencil on $E$, then each $C'\in \PP_b$ is nodal at $z$.  If $\mathfrak{p}_b=\{u_1+u_2\}$ consists of a fixed divisor, then $\PP_b$ contains a unique curve $C_z$ having multiplicity at least $3$ at $z$. If $u_1\neq u_2$, all other curves $C'\in \PP_b\setminus \{C_z\}$ are nodal at $z$, whereas if $u_1=u_2$, then all such $C'$ are cuspidal at $z$.

\vskip 5pt

Both possibilities {\bf{(P1)}} and {\bf{(P2)}} can be ruled out by a parameter count that contradicts the generality of the pair $(C,A)\in \mathcal{H}_{15,9}$ we started with. We first rule out {\bf{(P1)}}. Assume $C'\in \PP_b$ has a triple point at $z$ and no further singularities and denote by $\nu\colon \bar{C}\rightarrow C'$ the normalization. Set $\{z_1, z_2, z_3\}=\nu^{-1}(z)$ and $\bar{A}:=\nu^*\bigl(\OO_{C'}(C'))\in W^1_9(\bar{C})$. Since $\bar{A}$ is induced from a pencil of curves with a triple point at $z$, it follows that $\bigl|\bar{A}(-3z_1-3z_2-3z_3)\bigr|\neq
\emptyset$, therefore for degree reason $\bar{A}=\OO_{\bar{C}}(3z_1+3z_2+3z_3)$. We denote by $\H_{12,9}^{\mathrm{triple}}$ the Hurwitz space classifying degree $9$ covers
$\bar{C}\rightarrow \PP^1$ having a divisor of the form $3(z_1+z_2+z_3)$ in a fibre, where $\bar{C}$ is of genus $12$. Then $\H_{12,9}^{\mathrm{triple}}$ is pure of dimension
$\mathrm{dim}(\cM_{12})-1=32$. Let $\mathcal{Y}_{1}$ be the parameter space of pairs $(S, C')$, where $S\subseteq \PP^6$ is a smooth complete intersection of $4$ quadrics and $C'\subseteq S$ is an integral curve of arithmetic genus $15$ with a triple point as described by {\bf{(P1)}}. Let
$$
\begin{CD}
\cS @<\pi_1<< \mathcal{Y}_1 @>\pi_2>> \H_{12,9}^{\mathrm{triple}} \\
\end{CD}
$$
be the projections given by $\pi_1\bigl([S,C']\bigr):=[S]$ and $\pi_2\bigl([S,C']):=[\bar{C}, \bar{A}]$ respectively. With the notation above, from the adjunction formula $\nu^*(\OO_{C'}(1))=\omega_{\bar{C}}(-z_1-z_2-z_3)$. The fibre $\pi_2^{-1}\bigl(\pi_2([S,C'])\bigr)$ corresponds then to the choice of a $7$-dimensional space of sections $V\subseteq H^0\bigl(\bar{C}, \omega_{\bar{C}}(-z_1-z_2-z_3)\bigr)$ satisfying
$\mbox{dim}\Bigl(V\cap H^0\bigl(\omega_{\bar{C}}(-2z_1-2z_2-2z_3)\bigr)\Bigr)\geq 6$. Since $h^0\bigl(\omega_{\bar{C}}(-2z_1-2z_2-2z_3)\bigr)=6$, it follows that
$$\frac{V}{H^0\bigl(\omega_{\bar{C}}(-2z_1-2z_2-2z_3)\bigr)}\in \PP\Bigl(\frac{H^0\bigl(\omega_{\bar{C}}(-z_1-z_2-z_3)\bigr)}{H^0\bigl( \omega_{\bar{C}}(-2z_1-2z_2-2z_3)\bigr)}\Bigr)\cong \PP^2.$$
Therefore $\mathrm{dim}(\mathcal{Y}_1)=\mbox{dim}(\H_{12,9}^{\mathrm{triple}})+2=34\leq 39$, so we can invoke Proposition \ref{dims} to conclude that $\overline{\pi_1(\mathcal{Y}_1)}\neq \cS$ and rule out possibility {\bf{(P1)}}.

 \vskip 5pt

Next we rule out possibility {\bf{(P2)}}, focusing on the case when each $C'\in \PP_b$ is cuspidal at $z$. Passing to the normalization $\nu\colon \bar{C}\rightarrow C'$, setting $\bar{z}:=\nu^{-1}(z)$ we obtain that $\bar{A}:=\nu^*(\OO_{C'}(C')\in W^1_9(\bar{C})$ verifies $h^0\bigl(\bar{C}, \bar{A}(-4\bar{z})\bigr)\geq 1$. Let $\H_{14,9}^{\mathrm{four}}$ be the Hurwitz space classifying degree $9$ covers $\bar{C}\rightarrow \PP^1$ containing a divisor of type $4\bar{z}$ in one of its fibres and where $\bar{C}$ has genus $14$. Then $\H_{14,9}^{\mathrm{four}}$ is irreducible of dimension $39=\mbox{dim}(\cM_{14})$. Let $\mathcal{Y}_{2}$ be the parameter space of pairs $(S, C')$, where $S\subseteq \PP^6$ is a smooth complete intersection of $4$ quadrics and $C'\subseteq S$ is an integral curve of arithmetic genus $15$ with a cusp at $z$ as described by {\bf{(P2)}}. We consider the projections
$$
\begin{CD}
\cS @<\pi_1<< \mathcal{Y}_2 @>\pi_2>> \H_{14,9}^{\mathrm{four}} \\
\end{CD}
$$
given by $\pi_1\bigl([S,C']\bigr):=[S]$ and $\pi_2\bigl([S,C']):=[\bar{C}, \bar{A}]$ respectively. Observe that $\pi_2$ is birational onto its image. Indeed, given $[\bar{C}, \bar{A}]\in \pi_2(\mathcal{Y}_2)$, then we denote by $C'$ the image of the map $\varphi_{\omega_{\bar{C}}(2y)\otimes \bar{A}^{\vee}}\colon \bar{C}\rightarrow \PP^6$, in which case the canonical surface $S$ is recovered by (\ref{cansurf}). We conclude by Proposition \ref{dims} again that $\pi_1(\mathcal{Y}_2)$ is not dense in $\cS$. The final case when all curves $C'\in \PP_b$ are (at least) nodal at $z$ is ruled out  analogously.
\end{proof}

Before stating our next result, recall that one sets $\delta_i:=[\Delta_i]\in CH^1(\mm_g)$ for $0\leq i\leq \lfloor \frac{g}{2}\rfloor$, where $\Delta_i$ denotes the corresponding boundary divisor on $\mm_g$, see also \cite{HM} and \cite{FP}. We denote as usual by $\lambda\in CH^1(\mm_g)$ the Hodge class.   Recall also the formula \cite{HM} for the canonical class of $\mm_g$:
\begin{equation}\label{canm}
K_{\mm_g} \equiv 13\lambda-2\delta_0-3\delta_1 -2\delta_2-\cdots -2\delta_{\lfloor \frac{g}{2}\rfloor}\in CH^1(\mm_g).
\end{equation}

\begin{proposition}\label{numbers}
The rational curve $\Gamma$ is a sweeping pencil for the boundary divisor $\Delta_0$. Its intersection numbers with the standard generators of $CH^1(\mm_{16})$ are as follows:
$$\Gamma\cdot \lambda=22, \ \ \  \  \Gamma\cdot \delta_0=143,\ \  \ \Gamma\cdot \delta_j=0\   \  \mbox{ for } j=2, \  \ldots, 8.$$
\end{proposition}
\begin{proof}
First we construct a fibration whose moduli map is precisely the rational curve $m\colon \PP^1\rightarrow \mm_{15,2}$ considered in (\ref{penc3}).
We consider the curve $C\subseteq S$ and observe that since $\OO_C(C)\cong A\in W^1_9(C)$, we have that $C^2=9$, that is, the pencil $\bigl|\mathcal{I}_{\{x,y\}}(C)\bigr|$ has precisely $9$ base points, namely $x, y$, as well as the $7$ further points lying in the same fibre of the pencil $|A|$ as $x$ and $y$.  We consider the blow-up surface
$\epsilon\colon \widetilde{S}=\mathrm{Bl}_9(S)\rightarrow S$ at these $9$ points. It comes equipped with a fibration 
$$\pi\colon \widetilde{S}\longrightarrow \PP^1,$$ 
as well as with two sections $E_x, E_y\subseteq \widetilde{S}$ corresponding to the exceptional divisors at $x$ and $y$ respectively.

\vskip 4pt

In order to compute the intersection numbers of $R=m(\PP)$ with the tautological classes on $\mm_{15,2}$, we use for instance \cite{Tan}. The subscript indicates the moduli space on which the intersection number is computed.
\begin{equation}\label{eq:lam15}
\bigl(R\cdot \lambda\bigr)_{\mm_{15,2}}=\chi(\widetilde{S}, \OO_{\widetilde{S}})+g-1=h^2(S, \OO_S)+g=h^0(S,\OO_S(1))+15=22.
\end{equation}
Here we have used $H^1(\widetilde{S}, \OO_{\widetilde{S}})=H^1(S,\OO_S)=0$, as well as the fact that $S$ is a canonical surface, hence $\omega_S=\OO_S(1)$, therefore $h^2(\widetilde{S}, \OO_{\widetilde{S}})=h^2(S, \OO_S)=7$. Furthermore, recalling that all curves in the fibres of $m$ are irreducible, we find via \cite{Tan} that
$$\bigl(R\cdot \delta_0\bigr)_{\mm_{15,2}}=c_2(\widetilde{S})+4(g-1)=c_2(\widetilde{S})+56.$$

From the Euler formula, $c_2(\widetilde{S})=12\chi(\widetilde{S}, \OO_{\widetilde{S}})-K_{\widetilde{S}}^2$. We have already computed that
$\chi(\widetilde{S}, \OO_{\widetilde{S}})=8$, whereas $K_{\widetilde{S}}^2=K_S^2-9=\mbox{deg}(S)-9=7$, for $S$ is an intersection of $4$ quadrics.
Thus $c_2(\widetilde{S})=12\cdot 8-7=89$,  leading to $\bigl(R\cdot \delta_0\bigr)_{\mm_{15,2}}=89+4\cdot 14=145$.

\vskip 4pt

If we denote by $\psi_x,\psi_y\in CH^1(\mm_{15,2})$ the cotangent classes corresponding to the marked points labelled by $x$ and $y$ respectively,
we compute furthermore
$$R\cdot \psi_x=-E_x^2=1  \ \ \mbox{ and } \  \ R\cdot \psi_y=-E_y^2=1.$$
We now pass to the pencil $\xi\colon \PP^1\rightarrow \mm_{16}$ obtained from $m$ by identifying pointwise the disjoint sections $E_x$ and $E_y$ on the surface $\widetilde{S}$. First, using (\ref{eq:lam15}) we observe that
$$\Gamma\cdot \lambda=\xi(\PP)\cdot \lambda=\bigl(R\cdot \lambda\bigr)_{\mm_{15,2}}=22.$$
Furthermore, using Proposition \ref{prop:irr} we conclude that $\Gamma \cdot \delta_i=0$ for $i=1, \ldots, 8$. Finally, invoking for instance \cite[page 271]{CR3}, we find that
$$\Gamma\cdot \delta_0=\bigl(R\cdot \delta_0\bigr)_{\mm_{15,2}}-\bigl(R\cdot \psi_x\bigr)_{\mm_{15,2}}-\bigl(R\cdot \psi_y\bigr)_{\mm_{15,2}}=145-2=143.$$
\end{proof}

\vskip 2pt

\noindent \emph{Proof of Theorem \ref{sweep}.} Since the image of $m$ passes through a general point of $\mm_{15,2}$, the rational curve $\Gamma\subseteq \mm_{16}$ constructed in Proposition \ref{numbers} is a sweeping curve for the boundary  divisor $\Delta_0$. Using the expression (\ref{canm}) for the canonical divisor of $\mm_{16}$, we compute
$\Gamma\cdot K_{\mm_{16}}=13\ \Gamma \cdot \lambda-2\ \Gamma\cdot \delta_0=13\cdot 22-2\cdot 143=0$. Also $\Gamma\cdot \Delta_0=143>0$.
\hfill $\Box$

\section{Non-abelian Brill-Noether theory and $\mm_{16}$}

Based on the work of Mukai \cite{Mu} and Bertram--Feinberg \cite{BF}, one expects a general curve of genus $16$ to possess a finite number of stable vector bundles $E$ of rank $2$ with $\mbox{det}(E)=\omega_C$ and $h^0(C,E)=9$. Establishing this expectation is a prerequisite for the construction of the non-abelian Brill--Noether divisor $\mathfrak{D}_2$ on $\mm_{16}$.  

\vskip 3pt

We work in a more general context and fix a general curve $C$ of genus $g$ and a point $p\in C$. The moduli space $SU_C(2,\omega)$ of semistable vector bundles of rank $2$ with canonical determinant being singular, it is ill-suited to intersection-theoretic calculations. To determine the fundamental class of the non-abelian Brill-Noether locus $SU_C(2,\omega, k)$, as explained in \cite{Mu}, or \cite[\S8]{FJP1}, we pass instead to the Hecke correspondence over  the non-singular moduli space $SU_C(2,\omega_C(p))$ of stable rank $2$ vector bundles $G$ on $C$ with determinant $\mbox{det}(G)\cong \omega_C(p)$.

We choose  a universal rank $2$ vector bundle $\G$ on $C\times SU_C\bigl(2, \omega_C(p)\bigr)$ and consider the \emph{Hecke correspondence}
\[
\PPP:=\PP\Bigl(\G_{|\{p\}\times SU_C(2,\omega_C(p))}\Bigr),
\]
equipped with the projection $\pi_1\colon \PPP\rightarrow SU_C(2,\omega_C(p))$. The Hecke correspondence ist the parameter space of elementary transformations of the form
\begin{equation}\label{eq:hecke}
0\longrightarrow E\longrightarrow G\longrightarrow  \mathbb \ell_{p}\longrightarrow 0,
\end{equation}
where $G\in SU_C(2,\omega_C(p))$ and $\ell_p\in \PP\bigl(G_{|p}^{\vee}\bigr)$. One also has a morphism $\rho\colon \PPP\rightarrow 
SU_C(2,\omega)$ which assigns to the elementary transformation (\ref{eq:hecke}) the vector bundle $E$.

\vskip 4pt

Set $h:=c_1\bigl(\OO_{\PPP}(1)\bigr)$ and $\alpha:=c_1\bigl(\L_{\mathrm{odd}}\bigr)\in H^2\bigl(SU_C\bigl(2,\omega_C(p)\bigr),\mathbb Q\bigr)$, where
$\L_{\mathrm{odd}}$ is the ample generator of $\mbox{Pic}\bigl(SU_C\bigl(2,\omega_C(p))\bigr)\cong \mathbb Z$. Following both \cite[Section 5]{Mu} and \cite[\S 6]{FJP1}, for each $k\in\mathbb Z_{>0}$, the Lagrangian degeneracy locus 
\[
B_{\PPP}(k):=\Bigl\{\bigl[0\longrightarrow E\longrightarrow G\longrightarrow  \ell_{p}\rightarrow 0\bigr]\in \PPP: h^0(C, E)\geq k\Bigr\}
\]
has expected codimension $\beta(2,g,k)+1=3g-2-{k+1\choose 2}$. Note that $$B_{\PPP}(k)=\rho^{-1}\bigl(SU_C(2,\omega, k)\bigr).$$  
We consider the K\"unneth decomposition of the Chern classes of $\G$.  Using the fact that $\mbox{det}(\G)\cong \omega_C(p)\boxtimes \L_{\mathrm{odd}}$, we can  write accordingly
\begin{equation*}%\label{ef:chern_univ}
c_1(\G)=\alpha+(2g-1)\varphi \ \ \mbox{ and } \ \ c_2(\G)=\frac{\alpha^2-\beta}{4}+\psi+g \alpha\otimes \varphi,
\end{equation*}
where $\varphi\in H^2(C,\mathbb Q)$ is the fundamental class of the curve, $\beta \in  H^4\bigl(SU_C(2,\omega_C(p)), \mathbb Q\bigr)$
and $\psi \in H^3\bigl(SU_C(2,\omega_C(p)), \mathbb Q\bigr)\otimes H^1(C,\mathbb Q)$. We also consider the class
\[
\gamma \in H^6\bigl(SU_C(2,\omega_C(p)),\mathbb Q\bigr)
\]
by the formula
$\psi^2=\gamma\otimes \varphi $. One has the fundamental relation
\begin{equation}\label{eq:quadratic}
h^2=\alpha h-\frac{\alpha^2-\beta}{4}\in H^4(\PPP,\mathbb Q),
\end{equation}
from which we can recursively determine all powers of $h$. As explained in \cite[Proposition 6.4]{FJP1}, we have the following formulas in $H^*(\PPP,\mathbb Q)$ to be used repeatedly:
\begin{align*}
h^n=\frac{h(-2\alpha+2h)\sqrt{\beta}+\alpha^2-2\alpha h+\beta}{\sqrt{\beta}(\alpha^2-\beta)} \Bigl(\frac{\alpha+\sqrt{\beta}}{2}\Bigr)^n+\\
 \frac{h(2\alpha-2h)\sqrt{\beta}+\alpha^2-2\alpha h+\beta}{\sqrt{\beta}(\alpha^2-\beta)} \Bigl(\frac{\alpha-\sqrt{\beta}}{2}\Bigr)^n.
\end{align*}

Following \cite[\S 6]{FJP1}, we endow $B_{\PPP}(k)$ with the structure of a Lagrangian degeneracy locus. When $3g-3={k+1\choose2}$, we expect $B_{\PPP}(k)$ to be $1$-dimensional.
The virtual class of this Lagrangian degeneracy loci can be computed following \cite[Proposition 1.11]{Mu} and we find the following determinantal formula
\begin{equation}\label{eq:bp8}
[B_{\PPP}(k)]^{\mathrm{vir}}=\Delta_{k,k-1, \ldots, 1}(c_i)=\mathrm{det}\bigl(c_{k+j+1-2i}\bigr)_{i,j=1,\ldots, k}
\end{equation}
where $c_i\in H^{2i}(\PPP, \mathbb Q)$ are classes defined recursively by the following formulas, see \cite[Corollary 4.2]{LNP} and \cite[\S 6]{FJP1}:
\begin{equation}\label{eq:recursion_initial}
c_0=2, \ c_1=h, \ c_2=\frac{h^2}{2}, \ c_3=\frac{1}{3}\Bigl(\frac{h^3}{2}+\frac{\beta h}{4}-\frac{\gamma}{2}\Bigr),\
c_4=\frac{1}{4}\Bigl(\frac{h^4}{6}+\frac{\beta h^2}{3}-\frac{2\gamma h}{3}\Bigr),
\end{equation}
and for each $n\geq 1$,
\begin{equation}\label{eq:recursion}
(n+4)c_{n+4}-\frac{n+2}{2}\beta c_{n+2}+\Bigl(\frac{\beta}{4}\Bigr)^2 nc_n=hc_{n+3}-\Bigl(\frac{\beta h}{4}+\frac{\gamma}{2}\Bigr)c_{n+1}.
\end{equation}

In order to evaluate the determinant (\ref{eq:bp8}), we apply the recursion for the classes $h^n$ together with Thaddeus' formulas \cite{Th} determining all top intersection numbers of tautological classes on $SU_C(2,\omega_C(p))$. Precisely, for $m+2n+3p=3g-3$, one has
%\begin{equation}\label{eq:thaddeus}
$$
\int_{SU_C(2, \omega_C(p))} \alpha^m \cdot \beta^n \cdot \gamma^p=(-1)^{g-p}\frac{g! m!}{(g-p)! q!}2^{2g-2-p}(2^q-2)B_q,
$$
where $q=m+p+1-g$ and $B_q$ denotes the $q$-th Bernoulli number defined via the generating series 
$$\frac{x}{e^x-1}=1-\frac{x}{2}+\sum_{q \ \mathrm{ even}} B_q\cdot \frac{x^q}{q!}.$$ When the virtual dimension of $B_{\PPP}(k)$ equals $1$, we write
\[
[B_{\PPP}(k)]^{\mathrm{vir}}=F_k(\alpha, \beta, \gamma)+h\cdot G_k(\alpha, \beta, \gamma),
\]
where $F_k(\alpha, \beta, \gamma)$  and $G_k(\alpha, \beta, \gamma)$ are homogeneous polynomials of degree $3g-3$ and $3g-4,$ respectively. Moreover, $B_{\PPP}(k)$ is a $\PP^1$-fibration over $SU_C(2,\omega, k)$.  We write
$$
[B_{\PPP}(k)]^{\mathrm{vir}}=N_g \rho^*([E_0]),
$$
where $[E_0]\in SU_C(2,\omega)$ is general. Intersecting both sides of this equality with the class $\alpha$, we find an expression of $N_g$ as an integral over the smooth variety $SU_C(2, \omega_C(p))$: 
%$2N_g=h\cdot \alpha G_k(\alpha, \beta, \gamma)=-h\cdot F_k(\alpha, \beta, \gamma)$, hence
\begin{equation}\label{eq:int_p}
N_g=\bigl|SU_C(2, \omega, k)\bigr|=\frac{1}{2} \int_{SU_C(2, \omega_C(p))} \alpha\cdot G_k(\alpha, \beta, \gamma).
\end{equation}
Finding a closed formula looks challenging, but the values of $N_g$ can be computed for any given $g$ and we find that
$$N_6=N_8=1,\  N_{13}=3, \ N_{16}=11, \ \ N_{23}=11268, \ \ N_{27}=2461619, \ N_{36}=9522672213896.$$
The first two entries confirm Mukai's structure theorems \cite{Mu1} for general curves of genus $6$ and $8$ respectively, while the value of $N_{13}$ has been established in \cite{FJP1}.

\vskip 5pt

\noindent \emph{Proof of Theorem \ref{thm:existence}}. We consider the forgetful map $\vartheta\colon \mathcal{SU}_{16}(2, \omega, 9)\rightarrow \cM_{16}$ defined in (\ref{eq:varth}). Applying \cite{Te}, the Mukai-Petri map (\ref{eq:mp})
$$
\mu_E\colon \mbox{Sym}^2 H^0(C,E)\longrightarrow H^0(C, \mbox{Sym}^2 E)
$$
is an isomorphism for each bundle $E\in SU_C(2,\omega, 9)$ over a general curve $[C]\in \cM_{16}$. As explained in \cite{Mu}, this implies that, if non-empty, the non-abelian Brill-Noether locus $SU_C(2, \omega, 9)$ is smooth of dimension zero. On the other hand, the calculation above shows that $\mbox{deg}(\vartheta)=N_{16}=11$, in particular $SU_C(2,\omega, 9)\neq \emptyset$.  Therefore the general fibre $\vartheta^{-1}([C])$ consists of $11$ distinct stable vector bundles on $C$. 

\vskip 3pt

Next we show that each $E\in SU_C(2,\omega,9)$ is globally generated. First we observe that $E$ has to be stable. Indeed, else there exists a line bundle $L\in \mbox{Pic}^{15}(C)$ such that one has an exact sequence
$$0\longrightarrow L\longrightarrow E\longrightarrow \omega_C\otimes L^{\vee}\longrightarrow 0,$$
from which we find $h^0(C,L)+h^0(C, \omega_C\otimes L^{\vee})\geq 9$, therefore $h^0(C,L)\geq 5$, that is, $W^4_{15}(C)\neq 0$. But using \cite{EH2}, the locus of curves $[C]\in \cM_{16}$ satisfying this property has codimension at least $2$ in $\cM_{16}$, therefore $E$ must be stable.

\vskip 3pt

Assume now by contradiction that $E$ fails to be globally generated at $p\in C$, that is, the evaluation map $H^0(C,E)\rightarrow E_{|p}$ has rank at most $1$. Then since the map $\mu_E$ introduced in (\ref{eq:mp}) is an isomorphism, it follows that the map $$H^0(C, \mbox{Sym}^2E)\longrightarrow \mbox{Sym}^2 E_{|p}$$ has rank at most $1$, that is, $h^0\bigl(C, \mbox{Sym}^2E(-p)\bigr)\geq h^0(C, \mbox{Sym}^2(E))-1=44$. If $\mbox{End}_0(E)$ denotes the vector bundle of traceless endomorphisms of $E$, then using the identification $\mbox{End}_0(E)\cong \mbox{Sym}^2(E)\otimes \omega_C^{\vee}$, we obtain from this that 
$$
h^0\bigl(C, \mbox{End}_0(E)(p)\bigl)\geq 2.
$$
The map $H^0\bigl(C, \mbox{End}_0(E)(p)\bigr)\rightarrow H^0(C,\OO_C(2p))\cong \mathbb C$ associating to an endomorphism $\varphi\colon E\rightarrow E(p)$ its determinant being a quadratic form,  there exists $0\neq \varphi\in \mbox{End}_0(E)$ with $\mbox{det}(\varphi)=0$. Since $\mbox{tr}(\varphi)=0$, it follows that $\varphi^2\colon E\rightarrow E(2p)$ is the zero morphism. Then $L:=\mbox{Ker}(\varphi)$ is a line subbundle of $E$ and the relation $\varphi^2=0$ implies the existence of an inclusion $E/L\hookrightarrow L(p)$, or equivalently, $H^0(C, L^2(p)\otimes \omega_C^{\vee})\neq 0$. From this it follows that $\mbox{deg}(L)\geq 15$, contradicting the assumption that $E$ is stable. We conclude that $E$ is globally generated. In the same way one argues that in fact $E$ is a very ample vector bundle, that is, the map $\PP_E\rightarrow \PP H^0(C,E)^{\vee}$ is an embedding.

\hfill $\Box$ 

\subsection{The twisted non-abelian Petri map.} We fix a smooth curve $C$ of genus $g$, an integer $a\in \mathbb Z_{>0}$ and denote by $\mathfrak{Bun}_2(C)$ the Artin stack of all rank $2$ vector bundles on $C$. We consider the Brill-Noether locus of pairs
$$\mathfrak{bn}_{a,k}:=\bigl\{[F,D]: F \in \mathfrak{Bun}_2(C), \  D\in C_{a}, \ \mbox{det}(F)\cong \omega_C(-D), \ h^0(C,F)\geq k\bigr\}.$$ 
For $F\in \mathfrak{Bun}_2(C)$ as above, we have $F(D)\cong F^{\vee}\otimes \omega_C$. We also have a natural inclusion 
$H^0(C, \omega_C(-D)) \subseteq H^0(C, \omega_C)\subseteq H^0\bigl(C, \mbox{End}(F)\otimes \omega_C\bigr)$. The Petri map 
\begin{equation}\label{eq:Petri_classic}
\mathfrak{pe}_F\colon H^0(C,F)\otimes H^0(C, F^{\vee}\otimes \omega_C) \longrightarrow H^0(C, \mathrm{End}(F)\otimes \omega_C)
\end{equation}
has the space $\bigwedge^2 H^0(C,F)$ as an obvious kernel and induces the \emph{twisted Petri map}
\begin{equation}\label{eq:twisted}
\mu_{F,D}\colon \frac{H^0(C,F)\otimes H^0(C, F(D))}{\bigwedge^2 H^0(C,F)}\longrightarrow \frac{H^0\bigl(C, \mathrm{End}(F)\otimes \omega_C\bigr)}{H^0(C, \omega_C(-D))}.
\end{equation}

\begin{proposition}\label{prop:twisted_bn} 
Each component of $\mathfrak{bn}_{a,k}$ has dimension at least $3g-3+k(k+a)+{k\choose 2}$. Moreover, $\mathfrak{bn}_{a, k}$ is smooth of this expected dimension at the point $[F, D]$ if $\mu_{F,D}$ is injective.
\end{proposition}
\begin{proof}
The dimension bound follows from \cite[Corollary 1.2]{Oss}, even though he reaches this bound in a different way. Infinitesimal deformations of  $F$ are parametrized by the vector space $H^1(C, \mbox{End}(F))$. The deformations of $F$ whose determinant is a special line bundle are given by  
$$\Bigl\{v\in H^1(C, \mathrm{End}(F)): \mathrm{tr}(v)\in \mbox{Im}\{H^0(\OO_D(D))\rightarrow H^1(C, \OO_C)\}\Bigr\},$$ where we identify $H^0(\OO_D(D))$ with the tangent space of the symmetric product $C_a$ of $C$.  It is standard that deformations of $F$ preserving all the sections in $H^0(C,F)$ are those in the kernel of the cup-product map 
$H^1(C, \mbox{End}(F))\longrightarrow H^0(C,F)^{\vee}\otimes H^1(C,F)$,
which is known to be dual to the Petri map. By dualizing, we find the identification 
$$T_{[F,D]}^{\vee}\bigl(\mathfrak{bn}_k\bigr) \cong \mathrm{Coker}\bigl(\mu_{F,D}),$$
from which the conclusion follows.    
\end{proof} 

\begin{remark} Observe that for $a=0$, that is, when $\mbox{det}(F)\cong \omega_C$, the map $\mu_{F, \emptyset}$ coincides with the Mukai-Petri map $\mu_F$ given by (\ref{eq:mp}). Note furthermore, that for a curve $C$ of genus $16$, when $k=7$ and $a=3$ the expected dimension of the locus $\mathfrak{bn}_{3,7}$ provided by Proposition \ref{prop:twisted_bn} equals $-1$. 
\end{remark}

\section{The universal trisecant locus on $\mm_{16}$.}

We now introduce the divisor $\mathfrak{D}_2$ considered in the Introduction and defined as the closure of the locus of those curves $[C]\in \cM_{16}$ possessing a stable vector bundle $F$ of rank $2$ with $h^0(C,F)\geq 7$ and with $\mbox{det}(F)\cong \omega_C(-D)$, where $D=x+y+z$ is an effective divisor of degree $3$ on $C$. The goal of this Section is to establish a connection between $\mathfrak{D}_2$ and the universal moduli space parametrizing  stable rank $2$ vector bundles with canonical determinant.  Recall that $\vartheta \colon \mathcal{SU}_{16}(2, \omega, 9)\rightarrow \cM_{16}$ is the modular degree $11$ cover discussed in Theorem \ref{thm:existence} and to ease notation we set $\mathcal{BN}:=\mathcal{SU}_{16}(2,\omega,9)$.

\vskip 3pt

We fix a general curve $C$ of genus $16$ and choose a  vector bundle $E\in SU_C(2,\omega,9)$. Since $E$  may be assumed to be globally generated and satisfies $h^0(C,E)=9$, the evaluation map $H^0(C,E)\otimes \OO_C\twoheadrightarrow E$ induces a map 
$$\phi_E\colon C\longrightarrow G(2, H^0(C,E)^{\vee}), \ \ \ p\mapsto E_{|p}^{\vee}.$$
We denote by $\PP_E\subseteq \PP H^0(C,E)^{\vee}\cong \PP^8$ the scroll corresponding to the vector bundle $E$ and by $\PP_{E,p}\subseteq \PP_E$ the ruling corresponding to the point $p\in C$.   

\begin{definition}\label{def:trisecant}
We define the \emph{trisecant divisor} $\mathfrak{Tr}$ of $\mathcal{BN}$ to be the closure of the locus consisting of pairs $[C, E]$ such that the scroll $\PP_E$ has a proper trisecant line, that is, a line  $L \in G(2, H^0(C,E)^{\vee})$ not contained in the scroll $\PP_E$, such that 
\begin{equation}\label{eq:trisecant}
L \cdot \PP_E=x'+y'+z', \ \ \mbox{where } x' \in \PP_{E,x}, y'\in \PP_{E_y}, z'\in \PP_{E,z}, \ \mbox{ with } \  x,y,z\in C \mbox{ distinct}.
\end{equation}
\end{definition}

The divisor $\mathfrak{Tr}$ is intimately related to the divisor $\mathfrak{D}_2$. We have the following:

\begin{proposition}\label{prop:d2=tris}
The loci $\mathfrak{D}_2$ and $\vartheta(\mathfrak{Tr})$ agree outside a subset of codimension larger than $2$ in $\cM_{16}$. They are both equal to the closure of the following locus 
\begin{align*}
\Bigl\{[C]\in \cM_{16}: \exists E\in SU_C(2,\omega, 9) \mbox{ and distinct points } x,y,z\in C \ \mbox{ such that } \\
\mbox{ the evaluation map } \ \mathrm{ev}_{x,y,z}\colon H^0(C,E)\longrightarrow E_{|x} \oplus E_{|y} \oplus E_{|z} \  \mbox{ is not surjective}\Bigr\}.
\end{align*} 
\end{proposition}

\begin{proof} We pick a general element $[C,E]\in \mathfrak{Tr}$ and distinct points $x,y,z\in C$ such that the trisecant line $L \subseteq \PP H^0(C,E)^{\vee}\cong \PP^8$ satisfies (\ref{eq:trisecant}). The points $x', y', z'$ determine quotients 
$E_{|x}\twoheadrightarrow \mathbb C_{x}$, $E_{|y}\twoheadrightarrow \mathbb C_{y}$, respectively   
$E_{|z}\twoheadrightarrow \mathbb C_{z}$ and  an elementary transformation 

\begin{equation}\label{eq:eltransf2}
0\longrightarrow F\longrightarrow E\xrightarrow{\qquad (x',y',z') \qquad} \mathbb C_x\oplus \mathbb C_y\oplus \mathbb C_z\longrightarrow 0.
\end{equation}
Then the collinearity of $x', y', z'$ in $\PP^8$ is equivalent to $h^0(C,F)\geq h^0(C,E)-2=7$. Therefore a point $[C,E]\in \mathfrak{Tr}$ gives rise to a rank two vector bundle $F$ on $C$ with $\mbox{det}(F)=\omega_C(-x-y-z)$ and $h^0(C,F)\geq 7$.

\vskip 3pt

Furthermore, the vector bundle $F$ is stable. Indeed, otherwise we choose a maximal destabilizing sheaf $A\hookrightarrow F$, therefore $d=\mathrm{deg}(A)>\frac{\mathrm{deg}(F)}{2}$, that is, $d\in \{14, 15\}$, for $E$ was assumed to be stable.  We then also have an exact sequence
$$0\longrightarrow A\longrightarrow E\longrightarrow \omega_C\otimes A^{\vee}\longrightarrow 0,$$
therefore $9=h^0(C,E)\geq h^0(C,A)+h^0(C, \omega_C\otimes A^{\vee})$. It follows that $A$ is a linear system on $C$ having  Brill-Noether number $<-2$. This condition,  following \cite{EH2},  singles out a locus of codimension larger than $2$ in $\cM_{16}$. We conclude that $F$ is stable.

\vskip 4pt

Conversely, let us start with a stable vector bundle $F$ on $C$ with $\mbox{det}(F)\cong \omega_C(-D)$, where $D=x+y+z$ and with $h^0(C,F)=7$. We first observe by Riemann-Roch that $h^0(C, F(D))= 10$. Let $Q_x:=F(D)_{|x}$, $Q_y:=F(D)_{|y}$ and $Q_z:=F(D)_{|z}$ and then  set $Q:=Q_x\oplus Q_y\oplus Q_z\cong \mathbb C^6$. Since $\mbox{det}(F)\cong \omega_C(x+y+z)$, the vector space $Q$ is endowed with an alternating form given by 
$$s\wedge t\mapsto \mbox{Res}_x(s\wedge t)+\mbox{Res}_y(s\wedge t)+\mbox{Res}_z(s\wedge t).$$
\vskip 3pt

We  write down the cohomology exact sequence
$$0\longrightarrow H^0(C,F)\longrightarrow H^0(C, F(D)) \stackrel{\mathrm{pr}} \longrightarrow Q\stackrel{\delta}\longrightarrow H^1(C,F)\longrightarrow \H^1(C,F(D))\longrightarrow 0.$$
Set $W:=\mathrm{pr}\bigl(H^0(C,F(D)\bigr)$. Since $h^0(C,F)=7$, we obtain that $\mbox{dim}(W)=3$ and  in fact $W$  is a maximal isotropic subspace of $Q$. Under our generality assumptions we may also assume $h^0(C, F(x))=h^0(C,F(y))=h^0(C, F(z))=7$. It follows there exists unique $1$-dimensional subspaces $\langle v_x\rangle\in \PP(Q_x)$, $\langle v_y\rangle \in \PP(Q_y)$ and $\langle v_z\rangle \in \PP(Q_z)$ such that $\delta(v_x)=\delta(v_y)=\delta(v_z)\neq 0$. We then form the following elementary transformation
$$0\longrightarrow E\longrightarrow F(x+y+z)\longrightarrow \frac{Q_x}{\langle v_x\rangle}\oplus  \frac{Q_y}{\langle v_y\rangle} \oplus \frac{Q_y}{\langle v_y\rangle}\longrightarrow 0.$$
Then  $\mbox{det}(E)\cong \omega_C$ and having $h^0(C,E)=h^0\bigl(C, F(x+y+z)\bigr)-1=9$. The proof that $E$ is stable is standard and we skip it.

\vskip 4pt

We are left with showing that both descriptions of $\mathfrak{D}_2$ and $\vartheta(\mathfrak{Tr})$ are equivalent to requiring that the curve $[C]\in \cM_{16}$ possess a vector bundle $E\in SU_C(2,\omega,9)$ for which there exist distinct points $x,y,z\in C$ such that the evaluation map $$H^0(C,E)\longrightarrow E_{|x}\oplus E_{|y}\oplus E_{|z}$$ is not surjective, or equivalently, the dual map $E_{| x}^{\vee}\oplus E_{| y}^{\vee} \oplus E_{| z}^{\vee}\rightarrow H^0(C,E)^{\vee}$ is not injective. A non-zero element in the kernel corresponds to  $(\ell_x, \ell_y, \ell_z)\in E_{| x}^{\vee}\times E_{| y}^{\vee}\times E_{| z}^{\vee}$ such that for every section $s\in H^0(C,E)$ one has 
\begin{equation}\label{eq:depend}
\ell_x\bigl(s(x)\bigr)+\ell_y\bigl(s(y)\bigr)+\ell_z\bigr(s(z)\bigr)=0.
\end{equation}

The scroll $\PP_E$ being embedded in $\PP H^0(C,E)^{\vee}\cong \PP^8$, it follows that $\ell_x, \ell_y, \ell_z$ are necessarily non-zero, in particular they give rise to points $x'=[\ell_x]\in \PP_{E,x}$, $y'=[\ell_y]\in \PP_{E,y}$ and $z'=[\ell_z]\in \PP_{E,z}$. In this case, the condition  (\ref{eq:depend}) is equivalent to the points $x', y', z'\in \PP^8$ being collinear, which is precisely the condition (\ref{eq:trisecant}). The reverse implication is also immediate.   
\end{proof}

\begin{remark}
The construction of Proposition \ref{prop:d2=tris} can be reformulated in terms of projective geometry as follows. Starting with $E\in SU_C(2,\omega, 9)$ as above, recall that $L\subseteq \PP^8$ is the proper trisecant line to the scroll $\PP_E$ satisfying (\ref{eq:trisecant}). Let $\tau\colon \PP^8\dashrightarrow \PP^6$ be the projection of center $L$ and note that $\tau$ contracts the rulings $\PP_{E,x}, \PP_{E,y}, \PP_{E,z}$ of $\PP_E$ to three points $x''\in \langle L, \PP_{E,x}\rangle\cap \PP^6$, $y''\in \langle L, \PP_{E,y}\rangle \cap \PP^6$ and $z''\in \langle L,\PP_{E,z}\rangle \cap \PP^6$ respectively. Set 
$$\PP_F:=\tau(\PP_E)\subseteq \PP^6$$
and clearly this is the scroll corresponding to the vector bundle $F$ defined by the elementary transformation (\ref{eq:eltransf2}). Conversely, starting from a scroll $\PP_F\subseteq \PP^6$ as above and with points $x''\in \PP_{F,x}$, $y''\in \PP_{F,y}$ and $z''\in \PP_{F,z}$,  denoting by $H_F\in \bigl|\OO_{\PP_F}(1)\bigr|$ a linear section, we consider the rational map 
$$\nu\colon \PP_F\dashrightarrow \PP^8$$
induced by the linear system $\bigl|\I_{\{x'',y'',z''\}/\PP_F}(H_F+\PP_{F,x}+\PP_{F_y}+\PP_{F,z})\bigr|$. Then $\nu$ is the inverse map of $\tau_{\PP_E}$, in particular $\nu(\PP_F)=\PP_E$.
\end{remark}

\subsection{The determinantal structure of $\mathfrak{D}_2$}\label{precise-def-d2}

We now construct the divisor $\mathfrak{D}_2$ as the push-forward under the map $\vartheta\colon \mathcal{BN}\rightarrow \cM_{16}$ of a certain degeneracy locus. Let $p\colon \cC\rightarrow \mathcal{BN}$ be the universal curve and denote by $\E$ the universal rank $2$  bundle over $\cC$ normalized in such a way that $\mbox{det}(\E)\cong \omega_p$. 
We introduce the  relative quot scheme of length $3$ 
\begin{equation}\label{eq:U}
\mathfrak{u} \colon \cU=\mathrm{Quot}^3_{p}(\E) \longrightarrow \mathcal{BN}.
\end{equation} 
A general element of $\cU$ is a tuple $\bigl(C,E,x+y+z, \tau\bigr)$, where $E\in SU_C(2,\omega, 9)$ is a vector bundle, $x,y,z\in C$ are distinct and $E\twoheadrightarrow \tau $ is a quotient over a  $0$-dimensional scheme of length $3$ supported on the divisor $x+y+z$. Therefore, $\tau$ is given by $1$-dimensional quotients
 $\ell_x\in \PP\bigl(E_{|x}^{\vee}\bigr)$, $\ell_y\in \PP\bigl(E_{|y}^{\vee}\bigr)$ and $\ell_z\in \PP\bigl(E_{|z}^{\vee}\bigr)$.  We form the fibre product:
\[
\xymatrixcolsep{8pc}
\xymatrix{
\cC_{\cU}=\cU\times_{\mathcal{BN}} \cC \ar[r]^{\pi_2} \ar[d]^{\pi_1}& \cC \ar[d]^{p}\\
\cU \ar[r]^{\mathfrak{u}} & \mathcal{BN}}
\]
One has a tautological exact sequence over $\cC_{\cU}$ 
$$0\longrightarrow \F \longrightarrow \pi_2^*\E \longrightarrow \mathcal{Q}_{\mathrm{univ}}\longrightarrow 0,$$
where $\mathcal{Q}_{\mathrm{univ}}$ is locally free.
We set $\cA:=\bigl(\pi_1)_*\bigl(\pi_2^*\E\bigr)\cong \mathfrak{u}^*(p_*\E)$. Note that $\cA$ is a reflexive sheaf of rank $9$ over $\cU$, locally free over the locus of points $(C,E,\tau)$ where $h^0(C,E)=9$. Next we introduce the locally free sheaf of rank $3$ over $\cU$
$$\cB:=(\pi_1)_*\bigl(\mathcal{Q}_{\mathrm{univ}}\bigr)$$ and we have an evaluation morphism
\begin{equation}\label{eq:global_ev}
\mathfrak{ev}\colon \cA\longrightarrow \cB
\end{equation}
whose fibre over a point $t:=\bigl(C, E, \tau\bigr)\in \cU$ is precisely the map $\mathrm{ev}$  considered in the Introduction via the morphism (\ref{eq:evxyz}).
\begin{definition}\label{def:d2}
We let $\cZ\subseteq \cU$ be the rank $2$ degeneracy locus of the morphism $\mathfrak{ev}$ and denote by $\cZ^{\circ}$ the subvariety of $\cZ$ of points $t\in \cU$ such that $\mbox{rk}(\mathfrak{ev}_t)=2$. We denote by $$\mathfrak{D}_2=\mathfrak{\vartheta\circ u}(\cZ).$$ 

%Finally, we denote by $\F:=\mbox{Ker}\bigl(\mathfrak{ev}_{|\cZ}\bigr)$ the corresponding kernel bundle. 

\end{definition}
Since $\mbox{dim}(\cU)=\mbox{dim}(\mathcal{BN})+6=51$, each component of $\cZ$ has dimension at least equal to $44=\mbox{dim}(\mathcal{BN})-1$. The fibre of $(\pi_1)_*(\F)$ over a point $t\in \cZ^{\circ}$  is canonically identified with $H^0(C,F)$, where the vector bundle $F$ is defined by an exact sequence
$$0\longrightarrow F\longrightarrow E\stackrel{\mathfrak{ev}_t}\longrightarrow \tau\longrightarrow 0.$$
For $t\in \cZ^{\circ}$, let $\mathrm{det}_F\colon \bigwedge^2 H^0(C,F)\rightarrow H^0\bigl(C,\omega_C(-x-y-z)\bigr)$ be the determinant map.

\begin{definition}\label{def:Z12}
Let $\cZ_{1}$ be the closed subvariety of $\cZ^{\mathrm{\circ}}$ consisting of points $t$ for which $\mathrm{det}_F$ fails to be surjective.
\end{definition}

By construction, each component of $\cZ_{1}$ has dimension at least $\mbox{dim}(\cZ)-9\geq 35$. We shall exhibit a component of $\cZ_1$ having dimension $35$. This implies the existence of a component of $\cZ$ of dimension $44$, which is responsible for the divisor $\dd2$.  

%For a general point $t=\bigl(C,E, x,y,z, \ell_x, \ell_y, \ell_z)\in \cZ_{2}$, if we construct the surface $X$ via (\ref{eq:X}), then note that the class $B\equiv H-C$ has all the attributes of a smooth curve of genus $4$, however in general $H^0(X,B)=0$. Requiring this class be effective, leads precisely to the linkage condition described in (\ref{subsect1}).

\section{The geometry of the divisor $\mathfrak{D}_2$}

The main goal of this Section is to prove that the divisor $\mathfrak{D}_2$ possesses a uniruled component covered by a family of $K_{\mm_{16}}$-negative rational curves, thus proving Theorem \ref{thm:bnr2}. To that end, one needs to establish the various transversality assumptions described in the Introduction. 

We start with a smooth curve $[C]\in \cM_{16}$ and with a very ample vector bundle $F$ with $h^0(C,F)=7$ and 
$\mbox{det}(F)\cong \omega_C(-x-y-z)$. As in the Introduction, we  denote by $\phi_F\colon C\rightarrow G(2, H^0(C,F)^{\vee}\bigr)\cong G(2,7)=:\GG$ the map given by $\phi_F(p):=F_{|p}^{\vee}$ for every $p\in C$. We consider furthermore the $12$-dimensional linear subspace 
$$\Pi:=\PP H^0\bigl(C, \omega_C(-x-y-z))^{\vee}\bigr) \subseteq \PP\Bigl(\bigwedge^2 H^0(C,F)^{\vee}\Bigr) \cong \PP^{20},$$ and introduce the following linear section of the Grassmannian
\begin{equation}\label{eq:X}
X:=\Pi\cdot \GG.
\end{equation}

If this intersection is transverse, then $X$ is a smooth surface and by the adjunction formula $\omega_X \cong \OO_X(1)$, that is, $X\subseteq \Pi\cong \PP^{12}$ is canonically embedded. We also compute $\OO_C(C)\cong \omega_C\otimes \mbox{det}(F)^{\vee}\cong \OO_C(x+y+z)$. From the exact sequence (\ref{eq:seq1}), we obtain that $\mbox{dim } \bigl|\OO_X(C)\bigr|=1$. We denote by $\Gamma_2\subseteq \mm_{16}$ the curve induced by this pencil on $X$ and determine its numerical parameters. From this it will follow that $\Gamma_2$ is a $K_{\mm_{16}}$-negative curve, which is the essential ingredient in the proof of Theorem \ref{main_thm}.

\begin{proposition}\label{prop:pencild2}
Assume there are only nodal curves in $\bigl|\OO_X(C)\bigr|$. Then the intersection numbers of $\Gamma_2$ with the standard generators of $CH^1(\mm_{16})$ are given by 
$$\Gamma_2\cdot \lambda=29, \ \ \ \Gamma_2\cdot \delta_0=189, \ \mbox{ and } \ \Gamma_2\cdot \delta_j=0 \ \ \mbox{ for } j=2, \ldots, 8.$$
\end{proposition}
\begin{proof}
Since the curve $C$ can be chosen not to be trigonal, $h^0(C, \OO_C(x+y+z))=1$, therefore since $C^2=3$, the base locus of $\bigl|\OO_X(C)\bigr|$ consists of the points $x, y, z$. We denote by $\widetilde{X}:=\mbox{Bl}_{\{x,y,z\}}(X)$ the blow-up of $X$ at these points and by $f\colon \widetilde{X}\rightarrow \PP^1$ the fibration induced by the proper transform of $\bigl|\OO_X(C)\bigr|$. Then $\Gamma_2\subseteq \mm_{16}$ is the moduli map corresponding to this fibration of stable curves.

\vskip 3pt

To compute the intersection numbers of $\Gamma_2$ we proceed along the lines of Proposition \ref{numbers}. First, since $H^1(\widetilde{X}, \OO_{\widetilde{X}})=0$ and $H^2(\widetilde{X}, \OO_{\widetilde{X}}) \cong H^2(X,\OO_X)\cong H^0(X, \OO_X(1))^{\vee}$,   we have 
$$\Gamma_2\cdot \lambda=\chi(\widetilde{X}, \OO_{\widetilde{X}})+g-1=h^2(X,\OO_X(1))+1+g-1=13+16=29.$$
Then, we use \cite{Tan}, to conclude via the Noether formula that  
$$\Gamma_2\cdot \delta_0=c_2(\widetilde{X})+4(g-1)=12\chi(X, \OO_X)-K_{\widetilde{X}}^2+60.$$
As we have seen,  $\chi(X, \OO_X)=14$, whereas  $K_{\widetilde{X}}^2=K_X^2-3=\mbox{deg}(G(2,7))-3=39$, therefore we compute that $\Gamma_2\cdot \delta_0 =12\cdot 14-39+60=189$, which finishes the proof.
\end{proof}

\subsection{Constructing a smooth linear section of $\GG$.}\label{subsect1} In order to be able to apply Proposition \ref{prop:pencild2} and ultimately establish Theorem \ref{thm:bnr2}, we need to exhibit a component of the trisecant divisor $\mathfrak{Tr}$ on $\mathcal{SU}_{16}(2,\omega,9)$, whose general point $[C, E]$ 
satisfied two properties:

\begin{itemize}
\item It induces via the intersection (\ref{eq:X}) a \emph{smooth} canonical surface $X=\Pi\cdot \GG$.
\item This construction deforms in such a way that it fills-up a divisor in $\cM_{16}$. This will turn out to be the divisor $\dd2$ described in the Introduction.
\end{itemize}

\vskip 4pt

We use to that end a linkage construction.  Let us start with a general curve $B$ of genus $4$ and with a general stable vector bundle $F$ on $B$ of rank $2$ and degree $15$. Then $H^1(B,F)=0$ and by Riemann-Roch we compute $h^0(B,F)=9$. Set 
$$L:=\mbox{det}(F)\in \mbox{Pic}^{15}(B),$$ therefore $h^0(B,L)=12$. Since $F$ is a globally generated, using e.g. \cite[Corollary 3.5]{PR}, we can choose a general $7$-dimensional subspace $V\subseteq H^0(B,F)$ such that the evaluation morphism 
$V\otimes \OO_B\rightarrow F$ remains surjective. We denote by 
$$\phi_{V}\colon B\longrightarrow G(2,V^{\vee})\cong \GG$$
the induced map, then let us introduce the determinant map 
\begin{equation}\label{eq:detV}
\mathrm{det}_V\colon \bigwedge^2 V\longrightarrow H^0(B,L).
\end{equation}

\begin{lemma}\label{lemma2+5}
There exists a triple $(B,F,V)$ as above such that $\mathrm{rk}(\mathrm{det}_V)=11$ and the map $V\otimes \OO_B\rightarrow F$ is surjective.
\end{lemma}
\begin{proof}
We pick general line bundles $A\in \mbox{Pic}^5(B)$ and $M\in \mbox{Pic}^{10}(B)$ and specialize $F$ to the direct sum $F_{\mathrm{spl}}:=A\oplus M$, therefore $L=A\otimes M$. Choose a $5$-dimensional subspace $W\subseteq H^0(B,M)$ and set $V_{\mathrm{spl}}:=H^0(A)\oplus W$. It is easy to show that the multiplication map $H^0(C,A)\otimes W\rightarrow H^0(L)$ is injective. Since $\bigwedge^2 H^0(C,A)\oplus \bigwedge^2 W\subseteq \mbox{Ker}\bigl(\mathrm{det}_{V_{\mathrm{spl}}}\bigr)$, it follows that $\mbox{rk}\bigl(\mathrm{det}_{V_{\mathrm{spl}}}\bigr)=10$. Note that $V_{\mathrm{spl}}\otimes \OO_B\rightarrow F_{\mathrm{spl}}$ is surjective. 

\vskip 3pt

We deform this example and  consider a general extension of the form 
$$0\longrightarrow A\longrightarrow F\stackrel{\mathrm{pr}}\longrightarrow M\longrightarrow 0,$$
where we vary $A\in \mbox{Pic}^5(B)$ and $M\in \mbox{Pic}^{15}(C)$. The map $\mathrm{pr}\colon H^0(B,F)\rightarrow H^0(B,M)$ is surjective and we choose a general $5$-dimensional subspace $W\subseteq H^0(B,M)$ and set $V:=\mathrm{pr}^{-1}(W)\subseteq H^0(B,F)$. Using the  map 
$$\mbox{Ext}^1(M,A)\longrightarrow \mathfrak{Bun}_2(C)$$
into the Artin stack of all rank $2$ vector bundles on $C$, we can deform the split bundle $F_{\mathrm{spl}}$ into a stable bundle $F$ with $\mbox{det}(F)\cong L$. To get $\mbox{rk}(\mathrm{det}_V)=11$, it suffices to choose a subspace $W\subseteq H^0(B,M)$ such that the map 
$$\bigwedge^2 W\longrightarrow \frac{H^0(B, A\otimes M)}{H^0(B,A)\otimes W}$$
has rank precisely one. Proving this is straightforward, but we can also refer directly to Proposition \ref{prop:transv3}. 
\end{proof}

\vskip 4pt

Using Lemma \ref{lemma2+5} we return to the initial setting and having $B$ and $F$, we choose a subspace $V\in G\bigl(7, H^0(B,F)\bigr)$ such that $\mbox{rk}(\mathrm{det}_V)=11$.  For such a choice,  the span of the Pl\"ucker image of $B$, that is,
$$\bigl\langle \phi_V(B)\bigr\rangle =\PP\bigl(\mbox{Im}(\mathrm{det}_V^{\vee})\bigr):=\Lambda$$ 
is a $10$-dimensional subspace of $\PP\bigl(\bigwedge^2 V^{\vee}\bigr)\cong \PP^{20}$. 

\vskip 4pt
We have $B\subseteq \GG\cdot \Lambda$ and the intersection $\GG\cdot \Lambda$ is transverse along $B$ if and only if $\mathbb{T}_p(B)=\mathbb{T}_p(\GG)\cap \Lambda$, for each $p\in B$. This condition in turn is equivalent to the surjectivity of the \emph{conormal morphism} of sheaves 
$$\epsilon\colon \mbox{Ker}(\mathrm{det}_V)\otimes \OO_B\cong \OO_B^{\oplus 10}\longrightarrow N_{B/\GG}^{\vee}(1),$$
where we use the canonical identification of stalks for each $p\in B$
$$\mbox{Coker}(\epsilon)_p\cong \left(\frac{\mathbb{T}_p(X)\cap \Lambda}{\mathbb{T}_p(B)}\right)^{\vee}\otimes \OO_B(1)_{p}.$$
Note that  $\mbox{Ker}(\mathrm{det}_V)\cong H^0\bigl(\PP^{20}, \I_{\Lambda/\PP^{20}}(1)\bigr)$.
\vskip 4pt

Since $\mbox{Im}(\mathrm{det}_V)\subseteq H^0(B,L)$ is a hyperplane, we can choose $\eta\in H^0(B,L)^{\vee}$ such that $\mbox{Ker}(\eta)=\mbox{Im}(\mathrm{det}_V)$. We then have an associated alternating form 
$$\chi_V\colon H^0(B,F)\otimes H^0(B, F) \longrightarrow \mathbb C, \ \ \ s\wedge t\mapsto \eta\bigl(\mathrm{det}_V(s\wedge t)\bigr).$$
Since, by definition, the $7$-dimensional subspace $V\subseteq H^0(B,F)$ is isotropic for $\chi_V$, it follows that $\mbox{rk}(\chi_V)\leq 4$.

\begin{definition}\label{def:transv}
We say that a triple $(B, F, V)$ as above is \emph{transverse} if $\mbox{rk}(\mathrm{det}_V)=11$, the evaluation map $\phi_V\colon B\rightarrow G(2,V^{\vee})$ is an embedding and the conormal morphism of sheaves
$\epsilon \colon \OO_B^{\oplus 10}\rightarrow N_{B/\GG}^{\vee}(1)$
is surjective.
\end{definition}

\begin{proposition}\label{eq:transv3}
There exists a transverse triple $(B,F,V)$ as above.
\end{proposition}
\begin{proof} We shall construct in Section \ref{sect:transv} one such triple, see in particular Propositions \ref{prop:transv3} and \ref{prop:transv4}.  Not to derail the flow of exposition, we shall proceed assuming this fact.
\end{proof}

\begin{proposition}\label{prop:excess}
For a transverse triple $(B,F,V)$ as above, we have the scheme-theoretical equality
$$\Lambda\cdot\GG=B+x+y+z,$$
where $x,y,z$ are three distinct points not lying on the curve $B$.
\end{proposition}
\begin{proof}
We apply the excess intersection formula \cite[Theorem 13.3]{EH3} for the non-proper intersection $\Lambda\cdot \GG$ with respect to the $1$-dimensional excess component $B$. The excess (obstruction) bundle $\mathfrak{Obs}$ has rank $1$ and is obtained from the following exact sequence 
$$0\longrightarrow N_{B/\GG} \longrightarrow N_{\Lambda/\PP^{20}}\otimes \OO_B \longrightarrow \mathfrak{Obs}\longrightarrow 0.$$
Clearly $N_{\Lambda/\PP^{20}}\otimes \OO_B \cong L^{\oplus 10}$, hence $\mbox{deg}\bigl(N_{\Lambda/\PP^{20}}\otimes \OO_B\bigr)=150$. From the exact sequence
$$0\longrightarrow T_B\longrightarrow T_{\GG|B}\longrightarrow N_{B/\GG}\longrightarrow 0,$$
since $K_{\GG}=\OO_{\GG}(-7)$, we compute that $\mbox{deg}(N_{B/\GG})=7\mbox{deg}(L)+2g(B)-2=111$. We thus obtain $\mbox{deg}(\mathfrak{Obs})=150-111=39$. Since the degree of the Grassmannian $\GG$ equals $42$, it follows that the residual scheme $\Lambda\cdot \GG\setminus B$ has length $\mbox{deg}(\GG)-\mbox{deg}(\mathfrak{Obs})=3$. That this is a smooth scheme follows from the fact that the triple $(B,F,V)$ was chosen to be transverse in the sense of Definition \ref{def:transv}.
\end{proof}

Having fixed a transverse triple $(B,F,V)$, we choose a general $12$-dimensional linear space 
\begin{equation}\label{eq:pi}
\Lambda\subseteq \Pi \subseteq \PP^{20}=\PP\Bigl(\bigwedge^2 V^{\vee}\Bigr),
\end{equation}
and the associated surface section $X=\GG\cdot \Pi$ introduced in (\ref{eq:X}). By Bertini's Theorem, coupled with the fact that $\mathbb{T}_p(B)=\mathbb{T}_p(\GG)\cap \Lambda$ for all points $p\in B$,  we obtain that $X$ is a smooth canonical surface containing $B$. There exists a pencil of $11$-dimensional linear spaces 
$\Lambda\subseteq P_t\subseteq \Pi$ containing $\Lambda$ and  we define the residual curve $C_t\subseteq X$ by the relation 
\begin{equation}\label{eq:C_t}
P_t\cdot X=B+C_t.
\end{equation}

We have the following result which leads to a point in the trisecant divisor $\mathfrak{Tr}$:

\begin{proposition}\label{prop:residual16} 
The pencil $\{C_t\}_{t\in \PP^1}$ on $X$ consists entirely of nodal curves of genus $16$ such that $\OO_{C_t}(1)\cong \omega_{C_t}(-x-y-z)$ for each $t\in \PP^1$. Moreover, the base locus of the pencil consists of the points $x,y,z\in X$.
\end{proposition}
\begin{proof}
Since $\Lambda\subseteq P_t$, it follows that $x,y,z\in C_t$ for each $t\in \PP^1$. From the adjunction formula on $X$ we compute $B^2=2g(B)-2-B\cdot K_X=6-15=-9$, therefore $$C_t^2=(H-B)^2=H^2+B^2-2H\cdot B=3.$$ This shows that $\mbox{Bs}\bigl|\OO_X(H-B)\bigr|=\{x,y,z\}$. Furthermore, we have 
$$H\cdot C_t=H^2-\mbox{deg}(B)=42-15=27,$$ from which we compute by applying again adjunction $2g(C_t)-2=C_t^2+H\cdot C_t=30$, therefore $g(C_t)=16$. 
\end{proof}

We now take the canonical surface $X$ corresponding to a transverse triple $(B,F,V)$ and a smooth curve $C\in \bigl|\OO_X(H-B)\bigr|$. Note that $\OO_C(1)\cong \omega_C(-x-y-z)$. We let $\cQ$ be the universal rank $2$ quotient bundle on $\GG$ and set
\begin{equation}\label{eq:def_FC}
F_C:=\cQ_{|C}.
\end{equation} 
Then $\mbox{det}(F_C)\cong \omega_C(-x-y-z)$ and $h^0(C,F_C)=7$. Before investigating whether $F_C$ is stable, we introduce the following concept following \cite{AFRW}:

\vskip 4pt

\begin{definition}\label{def:resonance}
Let  $\bigl(\mathcal{E}, U\bigr)$ be a coherent system on a variety $X$ consisting of a vector bundle $\E$ and a subspace of sections  $U\subseteq H^0(X,\E)$. If 
$\mathrm{det}_U\colon \bigwedge^2 U\rightarrow H^0(X, \bigwedge ^2 \E)$ is the determinant map, then  the \emph{resonance} of $(\mathcal{E}, U)$ is defined as
$$\cR(\mathcal{E}, U):=\PP\bigl(\mbox{Ker}(\mathrm{det}_U)\bigr)\cap G(2, U),$$
the intersection being taken inside $\PP\bigl(\bigwedge^2 U)$. We write $\cR(\mathcal{E}):=\cR\bigl(\mathcal{E}, H^0(X, \mathcal{E})\bigr)$. 
\end{definition}

\begin{lemma}\label{lemma:resonance}
If $\cR(F_C)=\emptyset$, then $F_C$ is a stable vector bundle.
\end{lemma}
\begin{proof}
Assume by contradiction $F_C$ is not stable and let $A$ be one of its maximal destabilizing line subbundles with $\mbox{deg}(A)\geq 14$. Since $\cR(F_C)=\emptyset$ it follows that $F_C$ has no subpencils, therefore $h^0(C,A)=1$ and one has an exact sequence
$$0\longrightarrow A\longrightarrow F_C\longrightarrow \omega_C\otimes A^{\vee}(-x-y-z)\longrightarrow 0.$$
Since $h^0(C,F_C)=7$, we have $h^0(C,\omega_C\otimes A^{\vee}(-x-y-z))\geq 6$, therefore by Riemann-Roch $h^0(C, A(x+y+z))\geq 8$, which contradicts the assumption $h^0(C,A)=1$. 
\end{proof}

\subsection{The geometry of the Hilbert schemes of curves in the Grassmannian.}

Recall that we defined via Definition \ref{def:d2} the subvariety $\cZ\subseteq \cU$ parametrizing points $$t=\bigl(C,E, x+y+z,\tau)$$ having a vector bundle $F=\mbox{Ker}\{E\twoheadrightarrow \tau\}$ with $\mbox{det}(F)\cong \omega_C(-x-y-z)$,  and satisfying $h^0(C,F)\geq 7$ and denoted by $\cZ_1\subseteq \cZ$ the subvariety characterized furthermore by the condition $\mbox{rk}(\mathrm{det}_F)\leq 11$. The geometric characterization of points $\cZ_1$ is transparent. Starting with $t\in \cZ$,  if we construct the surface $X$ via (\ref{eq:X}), the class 
$B\equiv H-C$ on $X$  has the attributes of a smooth curve of genus $4$, however in general $H^0(X,B)=0$. Requiring $B$ to be effective, leads precisely to the linkage of the genus $16$ curve $C$ to a genus $4$ curve $B$,  therefore to a point $t\in \cZ_1$.

\vskip 3pt

We denote by $\HH$ the Hilbert scheme of curves $B\hookrightarrow \GG\cong G(2,V^{\vee})$ of genus $4$ and degree $15$. Note that a triple $(B,F,V)$ like in Proposition \ref{eq:transv3} naturally leads to a point $[\phi_V\colon B \hookrightarrow \GG]\in \HH$, where $L=\phi_V^*(\OO_{\GG}(1))$.

\begin{proposition} The component of the Hilbert scheme $\HH$ corresponding to a transverse triple $(B,F,V)$  is smooth of dimension $84$.
\end{proposition}
\begin{proof} Having fixed $(B,F,V)$, we introduce the kernel bundle $M_V$  by the exact sequence 
$$0\longrightarrow M_V\longrightarrow V\otimes \OO_B\longrightarrow F\longrightarrow 0.$$ 
We have the identification  $T_{\GG|B}\cong M_V^{\vee}\otimes F$ of the restricted tangent bundle and write 
$$H^1\bigl(B, T_{\GG|B}\bigr)\cong H^1(B,M_V^{\vee}\otimes F)\cong H^0(B, M_V\otimes F^{\vee}\otimes \omega_B)^{\vee}=0,$$
where the last vanishing follows from the stability of $F$. From the usual exact sequence
$$0\longrightarrow T_B\longrightarrow T_{\GG|B}\longrightarrow N_{B/\GG}\longrightarrow 0,$$
it follows then by taking cohomology that $H^1(B, N_{B/\GG})=0$, therefore $\HH$ is smooth of dimension equal to $h^0(B, N_{B/\GG})=\mbox{deg}(N_{B/\GG})+\mbox{rk}(N_{B/\GG})\cdot (1-g(B))=111-27=84$, at the point $[B\hookrightarrow \GG]$.
\end{proof}

\vskip 3pt

Let us fix a point $[B\hookrightarrow \GG]\in \HH$ as above. We have the canonical identifications 
\begin{equation}\label{eq:ident6}
H^0\bigl(\GG, \I_{B/\GG}(1)\bigr)\cong \mathrm{Ker}(\mathrm{det}_V) \ \ \mbox{ and } \ H^1\bigl(\GG, \I_{B/\GG}(1)\bigr)\cong \mathrm{Coker}(\mathrm{det}_V).
\end{equation}

\begin{definition}\label{def:h1}
Let $\HH_1$ be the subvariety of $\HH$ defined by the condition 
$$H^1\bigl(\GG, \I_{B/\GG}(1)\bigr)\neq 0.$$
\end{definition}

Each component of $\HH_1$ has dimension at least $74=\mbox{dim}(\HH)-10$. For  a point $[B\hookrightarrow \GG]\in \HH_1$, we have $\mbox{dim}\langle B\rangle \leq 10$ and for a general point of $\HH_1$ we have equality. Next we introduce the parameter space which generically is a $\PP^9$-bundle over $\HH_1$ 
$$
\mathcal{P}:=\left\{\bigl([B\hookrightarrow \GG], P\bigr)\in \HH_1\times G\bigl(12,\bigwedge^2 V^{\vee}\bigr): P\in G\bigl(12,\bigwedge^2 V^{\vee}\bigr), \\  \langle B\rangle \subseteq P\subseteq \PP\bigl(\bigwedge^2 V^{\vee}\bigr)\right\}.
$$
Therefore $\mathcal{P}$ parametrizes $11$-dimensional subspaces $\PP^{11}\cong P\subseteq \PP\bigl(\bigwedge^2 V^{\vee}\bigr)$ containing the Pl\"ucker image of $B$. Each component of $\mathcal{P}$ has dimension at least $$9+\mbox{dim}(\HH_1)=83.$$ The automorphism group $\mbox{Aut}(\GG)\cong PGL(7)$ acts naturally on $\mathcal{P}$ and we have the following \emph{residuation map}
\begin{equation}\label{eq:residue}
\mathrm{res}\colon \mathcal{P}\dblq \mathrm{Aut}(\GG)\dashrightarrow \cZ_1, \ \ \bigl([B\hookrightarrow \GG], P\bigr)\mapsto [C, \cQ_{|C}],
\end{equation} 
where the curve $C$ of genus $16$ is defined by the equality $P\cdot \GG=B+C$. In a neighborhood of the point corresponding to a transverse triple $(B,F,V)$  this intersection is transverse and the map $\mathrm{res}$ is well-defined.  Note that the restriction $\cQ_{|C}$ of the tautological bundle is precisely the vector bundle $F_C$ on $C$ with $\mbox{det}(F_C)\cong \omega_C(-x-y-z)$ and $h^0(C,F_C)=7$.

\begin{proposition}\label{prop:linkage5}
The parameter space $\cZ$ has a generically smooth component of expected dimension $44$ provided one can exhibit a smooth component of dimension $83$ of $\HH_1$ corresponding to a transverse triple $(B,F,V)$.
\end{proposition}
\begin{proof}
We observe that the map $\mathrm{res}$ given by (\ref{eq:residue}) is locally an isomorphism. Therefore the existence of a component of $\HH_1$ of the expected dimension $74$ containing a point $[B\hookrightarrow \GG]$ corresponding of a transverse triple $(B,F,V)$ implies the existence of a generically smooth component of dimension $35=\mbox{dim}(\HH_1)+9-\mathrm{dim}(\mathrm{Aut}\bigl(\GG)\bigr)$ of $\cZ_1$. Around such a point $t\in \cZ_1$, we have the following estimate
$44\leq \mbox{dim}_t(\cZ)\leq \mbox{dim}_t(\cZ_1)+9=44$, thus finishing the proof.
\end{proof}

We now make the assumption in Proposition \ref{prop:linkage5} explicit. This will then be used in Section \ref{sect:transv}, where such a component of $\HH_1$ is explicitly constructed.

\begin{proposition}\label{prop:transv5}
A point $[B\hookrightarrow \GG]\in \HH_1$ for which $\mathrm{rk}(\mathrm{det}_V)=11$ corresponds to a smooth point of the expected dimension $74$ of $\HH_1$ if and only if the natural map
\begin{equation}\label{eq:nuB}
\nu_B\colon H^0(B, N_{B/\GG})\longrightarrow \mathrm{Hom}\bigl(\mathrm{Ker}(\mathrm{det}_V), \mathrm{Coker}(\mathrm{det}_V)\bigr)
\end{equation}
is surjective.
\end{proposition}
\begin{proof}
We recall the identifications (\ref{eq:ident6}). An element $\xi\in H^0(B, N_{B/\GG})$ gives rise to a morhphism $\theta_{\xi}\colon \I_{B/\GG}/{\I_{B/\GG}^2}\rightarrow \OO_B$. Let $\theta_{\xi}(1)\colon \I_{B/\GG}/{\I_{B/\GG}^2}(1)\rightarrow L$ be the morphism induced by $\theta_{\xi}$ after twisting by $L$.  A section $s\in \mbox{Ker}(\mathrm{det}_V)\cong H^0\bigl(\GG, \I_{B/\GG}(1)\bigr)$, induces a section $\overline{s}\in H^0\bigl(B, \I_{B/\GG}/\I^2_{B/\GG}(1)\bigr)$. The morphism $\nu_B$ is then obtained by  setting 
$$\nu_B(\xi)(s):=\theta_{\xi}(1)(\overline{s}) \ \mathrm{mod}\ \bigwedge^2 V\in H^0(B,L)/\bigwedge^2 V= \mathrm{Coker}(\mathrm{det}_V).$$ That the maximal rank of $\nu_B$ guarantees the smoothness and the expected dimension of $\HH_1$ at the point $[B\hookrightarrow \GG]$, follows from the general theory of degeneracy loci.
\end{proof}

In Section \ref{sect:transv} we shall construct a transverse triple $(B,F,V)$ for which the corresponding map $\nu_B$ defined by (\ref{eq:nuB}) is surjective. Assuming this for the time being, we proceed with the following conclusion:

\begin{theorem}\label{thm:44}
The variety $\cZ$ admits a generically smooth component of dimension $44$.
\end{theorem}  
\begin{proof} Follows directly from Proposition \ref{prop:linkage5}.
\end{proof}

\section{Transversal triples}\label{sect:transv}

We now explain the proof of Proposition \ref{eq:transv3}, that is, we construct a \emph{transverse} triple $(B,F,V)$ consisting of a smooth curve $B$ of genus $4$, a  vector bundle $F$ of rank $2$ on $B$ with $\mbox{det}(F)\cong L \cong A\otimes M$, where $A\in \mbox{Pic}^5(B)$ and $M\in \mbox{Pic}^{10}(B)$ and a $7$-dimensional space of sections $V\subseteq H^0(B,F)$ such that $\mathrm{rk}(\mathrm{det}_V)=11$.

\vskip 4pt

It will be enough to construct this example over a finite field, then lift it to characteristic zero. We work over the field $\kk=\mathbb F_{17}$. The curve $B$ of genus $4$ is given by the following equation of bidegree $(5,2)$ in $\PP^1\times \PP^1$:

$$
B=\Bigl\{\bigl([x_0,x_1], [y_0,y_1]\bigr) \in \PP^1\times \PP^1: a_0(x_0,x_1) y_0^2+a_1(x_0,x_1) y_0y_1+a_2(x_0,x_1) y_1^2=0\Bigr\},
$$
$$\mbox{where } \ a_0(x_0,x_1)=x_0^5+12x_0^4x_1+5x_0^3x_1^2+3x_0^2x_1^3+11x_0x_1^4+16x_1^5,$$
$$a_1(x_0,x_1)=7x_0^5+14x_0^4x_1+16x_0^3x_1^2+7x_0^2x_1^3+7x_0x_1^4+5x_1^5,$$
$$a_2(x_0,x_1)=13x_0^5+8x_0^4x_1+14x_0^3x_1^2+13x_0^2x_1^3+13x_0x_1^4+16x_1^5.$$

\vskip 4pt

One checks directly that $B\subseteq \PP^1\times \PP^1$ is a smooth curve of genus $4$. Denoting by $\pi_1, \pi_2\colon B\rightarrow \PP^1$ the two projections, note that $\mbox{deg}(\pi_1)=2$, therefore $B$ is hyperelliptic. We then set
$$A:=\pi_2^*\OO_{\PP^1}(1)\in \mbox{Pic}^5(B) \ \mbox{ and } \ \ M:=\pi_1^*\OO_{\PP^1}(5)\in \mbox{Pic}^{10}(B),$$ 

\noindent then choose the  bases $\bigl(u_0:=y_0, u_1:=y_1\bigr)$ of $H^0(B,A)$, respectively of $H^0(B,M)$ 
$$\bigl(m_0:=x_0^5, \ m_1:=x_0^4x_1, \ m_2:=x_0^3x_1^2, \ m_3:=x_0^2x_1^3, \ m_4:=x_0x_1^4, \ m_5:=x_1^5, \ m_6\bigr),$$
where $m_6$ does not lie in the image of the map $\pi_1^*\colon H^0(\PP^1, \OO_{\PP^1}(5))\rightarrow H^0(B,M)$ and is chosen in such a way that it satisfies both relations $y_0m_6=-a_1(x_0,x_1)y_0-2a_2(x_0,x_1)y_1$ and $y_1m_6=2a_0(x_0,x_1)y_0+a_1(x_0,x_1)y_1$ in the space $H^0(B,L)$.

\vskip 4pt

Via the Base Point Free Pencil Trick it follows that the multiplication map 
$$H^0(B,A)\otimes H^0(B,M)\longrightarrow H^0(B,L)$$
is surjective. We introduce the following basis of $H^0(B,L)$:
$$\bigl(\ell_i:=x_0^{5-i}x_1^iy_0, \ \ \ell_{6+i}:=x_0^{5-i}x_1^iy_1\bigr)_{i=0,\ldots,5}.$$
The choice of these three bases is retained throughout this section.

\subsection{The construction of the vector bundle $F$.}

We construct the rank $2$ vector bundle $F$ on $B$ sitting in an extension 
\begin{equation}\label{eq:extF}
0\longrightarrow A\longrightarrow F\stackrel{\pi_M}\longrightarrow M\longrightarrow 0.
\end{equation}
We introduce the $\kk$-vector space $H\cong \kk^9$ and denote by $(h_0, \ldots, h_8)$ its canonical basis. In order to define $F$, we first construct a map $\varphi\colon B \rightarrow G(2, H^{\vee})\subseteq \PP\bigl(\bigwedge^2 H^{\vee}\bigr)$ and then $F$ will be induced by $\varphi$ via the evaluation morphism. We introduce the following map
\begin{equation}\label{eq:alternating}
d\colon \bigwedge^2 H \longrightarrow H^0(B,L), 
\end{equation} given by 
$d\bigl(h_0\wedge h_1\bigr)=0$, $d\bigl(h_0\wedge h_{i+2}\bigr)=u_0\cdot m_i$ and $d\bigl(h_1\wedge h_{i+2}\bigr)=u_1\cdot m_i$, for $1\leq i\leq 6$,
whereas for $i<j$ with $i, j=0, \ldots, 6$, we write 
$$d\bigl(h_{i+2}\wedge h_{j+2}\bigr)=\sum_{a=0}^{11} c_{ij,a}\cdot \ell_a\in H^0(B,L),$$ 
where the $21\times 12$ matrix of coefficients $c_{ij,a}$ has the following form:

$$\bigl(c_{ij,a}\bigr)_{i<j, a=0, \ldots,11}=
\left(
\begin{array}{rrrrrrrrrrrr}
11&10&0&0&0&0&7&7&0&0&0&0\\
15&14&10&0&0&0&1&11&7&0&0&0\\
8&0&14&10&0&0&0&6&11&7&0&0\\
15&14&0&14&10&0&5&10&6&11&7&0\\
2&16&14&0&14&10&16&5&10&6&11&7\\
0&16&5&2&9&4&2&3&5&8&8&4\\
0&15&3&0&0&0&0&1&4&0&0&0\\
0&8&0&3&0&0&0&0&6&4&0&0\\
0&15&14&0&3&0&0&5&10&6&4&0\\
0&2&16&14&0&3&0&16&5&10&6&4\\
12&8&9&11&4&11&16&2&6&2&14&9\\
0&0&8&2&0&0&0&0&0&5&0&0\\
0&0&15&14&2&0&0&0&5&10&5&0\\
0&0&2&16&14&2&0&0&16&5&10&5\\
1&7&14&14&6&14&9&3&7&1&10&12\\
0&0&0&15&6&0&0&0&0&5&10&0\\
0&0&0&2&16&6&0&0&0&16&5&10\\
5&15&9&1&3&1&4&1&2&4&5&11\\
0&0&0&0&2&1&0&0&0&0&16&0\\
10&15&15&12&12&12&15&5&8&8&3&2\\
16&1&3&0&4&10&8&2&8&10&8&13
\end{array}
\right).
$$

The ordering of the $21$ rows is lexicographic. We have an induced morphism $\varphi\colon B\rightarrow \PP\bigl(\bigwedge^2 H^{\vee}\bigr)$ whose coordinates are the sections $d_{ij}:=d\bigl(h_i\wedge h_i\bigr)\in H^0(B,L)$, for $0\leq i<j\leq 8$. One checks directly that the Pl\"ucker identities 
$$d_{ij}\cdot d_{k\ell}-d_{ik}\cdot d_{j\ell}+d_{i\ell}\cdot d_{jk}=0\in H^0(B,L^2),$$
hold for $0\leq i<j<k<\ell\leq 8$ and that moreover the sections $d_{ij}$ have no common zero. It follows that $\varphi(B)\subseteq G(2, H^{\vee})$. Denoting by $\cQ_{H}$ the tautological rank $2$ bundle on $G(2, H^{\vee})$, we set $F:=\phi^*\cQ_{H}$. Then $F$ is a globally generated rank $2$ vector bundle with $\mbox{det}(F)\cong L$ and $H^0(B,F)\cong H$. Via this identification, the map $\phi_F\colon B\rightarrow G(2,H^0(B,F)^{\vee})$ induced by evaluation of global sections is precisely the map $\varphi$ we have just constructed. Similarly, the map $d$ described by (\ref{eq:alternating}) becomes the determinant map at the level of global sections of the vector bundle $F$, that is, 
$$\mathrm{det}_F\colon \bigwedge^2 H^0(B,F)\longrightarrow H^0(B,L).$$ Since $d\bigl(h_0\wedge h_1\bigr)=0$, the sections $h_0$ and $h_1$ generate the subpencil $A\hookrightarrow F$ and they get identified with the sections $u_0, u_1\in H^0(B,A)$. This also shows that the vector bundle $F$ sits in the extension (\ref{eq:extF}). Under the projection map $\pi_M\colon H^0(B,F)\twoheadrightarrow H^0(B,M)$, the section $h_{i+2}$ is mapped to $m_i$, for $i=0, \ldots, 6$.

\vskip 4pt

We then construct the $7$-dimensional space $V\subseteq H^0(B,F)\cong H$, as a lift under the map $\pi_M$ of the  $5$-dimensional subspace $W\subseteq H^0(B,M)$ generated by the following linearly independent sections:
$$
w_0=10m_0+14m_1+m_2+4m_3+5m_4+12m_5+3m_6,$$ 
$$w_1=11m_0+15m_1+4m_2+9m_3+4m_4+16m_5+13m_6,$$
$$w_2=16m_0+12m_1+5m_2+5m_3+2m_5+9m_6,$$
$$w_3=4m_0+11m_1+2m_2+2m_3+12m_4+12m_5+10m_6,$$
$$w_4=m_0+16m_1+9m_2+6m_4+7m_5+6m_6.$$

\vskip 4pt

\noindent Choosing lifts $\widetilde{w}_i\in H^0(C,F)$ of $w_i$, for $i=0, \ldots, 4$,  we set 
$$V:=\bigl\langle v_0:=h_0, v_1:=h_1, v_2:=\widetilde{w}_0, v_3:=\widetilde{w}_1, v_4:=\widetilde{w}_2, v_5:=\widetilde{w}_3, v_6:=\widetilde{w}_4\bigr\rangle \subseteq H^0(B,F).$$
\noindent Note that $V$ only depends on $W$ and not on the choice of the lifts $\widetilde{w}_i$. Recalling that the determinant map $\mbox{det}_V\colon \bigwedge^2 V\rightarrow H^0(B,L)$ has been defined via (\ref{eq:detV}), we have:

\begin{proposition}\label{prop:transv3} The following statements hold:
\begin{enumerate}
\item We have that $\mathrm{rk}(\mathrm{det}_V)=11$, therefore $\Lambda:=\PP\bigl(\mathrm{Im}(\mathrm{det}_V^{\vee})\bigr)$ is $10$-dimensional.
\item The map $\phi_V\colon B \rightarrow G(2,V^{\vee})$ is an embedding.
\end{enumerate}
\end{proposition}
\begin{proof}
It is immediate to show via the Base Point Pencil trick that the multiplication map $H^0(B, A)\otimes W\rightarrow H^0(B,L)$ is injective, therefore $\mathrm{rk}(\mathrm{det}_V)\geq 10$. Furthermore, $\mbox{Im}(\mathrm{det}_V)$ lies in the hyperplane of $H^0(B,L)$ defined by the following element in $H^0(B,L)^{\vee}$
\begin{equation}\label{eq:eta1}
\eta:=4\ell_0^{\vee}+2\ell_1^{\vee}+15\ell_2^{\vee}+4\ell_3^{\vee}+6\ell_4^{\vee}+13\ell_5^{\vee}+14\ell_6^{\vee}+7\ell_7^{\vee}+6\ell_8^{\vee}+12\ell_9^{\vee}+
\ell_{10}^{\vee}+13\ell_{11}^{\vee}.
\end{equation}
Since also $\mathrm{det}_V\bigl(\widetilde{w}_0\wedge \widetilde{w}_1\bigr)\notin H^0(B,A)\otimes W$, it follows that necessarily $\mbox{rk}(\mathrm{det}_V)=11$. For the second statement, it is enough to observe that the linear system $\bigl|\mbox{Im}\bigl(\mathrm{det}_V\bigr)\bigr|$ is very ample on $B$. Since the map induced by this linear system factors through $\phi_V$, it follows that also $\phi_V$ is  an embedding.
\end{proof}

We now show that the triple $(B,F, V)$ we just constructed is transverse. As before, $\GG=G(2,V^{\vee})$:

\begin{proposition}\label{prop:transv4}
For each $b\in B(\overline{\kk})$, one has $\mathbb T_b\bigl(\GG\bigr)\cap \Lambda=\mathbb T_b(B)$.
\end{proposition} 
\begin{proof}
As already explained, it suffices to show that the conormal morphism 
$$\epsilon\colon \mbox{Ker}(\mathrm{det}_V)\otimes \OO_B \longrightarrow N_{B/\GG}^{\vee}(1)$$ is surjective. We denote by $p_{ij}$ the Pl\"ucker coordinates on $\GG$ with respect to the basis $(v_0, \ldots, v_6)$ of $V$. The Grassmannian $\GG$ is cut out inside $\PP\bigl(\bigwedge^2 V^{\vee}\bigr)\cong \PP^{20}$ by the $35={7\choose 4}$ Pl\"ucker quadrics 
$q_{ijk\ell}=p_{ij}p_{k\ell}-p_{ik}p_{j\ell}+p_{i\ell}p_{jk}$,  for   $0\leq i<j<k<\ell\leq 6$. 
We form the $35\times 21$ Jacobian matrix $J_{\GG}$ of these quadrics. We denote furthermore by $T$ the $35\times 11$ matrix describing the inclusion $\mbox{Im}\bigl(\mathrm{det}_V^{\vee}\bigr)\hookrightarrow \bigwedge^2 V^{\vee}$ and we consider the product matrix
$$R:=J_{\GG}\cdot T\in M_{35,11}(\kk).$$
Since $B\subseteq \GG\cap \Lambda$, we have $\mbox{rk}\bigl(R(b)\bigr)\leq 9$. The proof is finished once it is established that $\mbox{rk}\bigl(R(b)\bigr)=9$, for all points $b\in B(\overline{\kk})$.

\vskip 4pt

To that end, we order the $35$ rows of $R$ lexicographically, thus starting with $0123$ and ending with $3456$. It turns out one needs to consider only three $9\times 9$ minors $R_1, R_2, R_3$ of $R$, which we write indexed by rows and columns as follows:
$$
R_1:=
\det
R[
0124,0125,0126,0135,0246,0256,1235,1345,1356;
\,0,3,4,5,6,7,8,9,10],
$$
$$
R_2:=
\det
R[
0146,0156,0234,0245,0246,1235,1236,2456,3456;
\,0,1,2,3,4,5,6,7,9],
$$
$$
R_3:=
\det
R[
0123,0126,0246,0356,1234,1235,1256,1345,1356;
\,0,1,3,4,5,6,7,8,10]
$$
One checks that for each $b\in B(\overline{\kk})$ at least one minor $R_i(b)\neq 0$, which shows that $\epsilon$ is a surjective morphism of sheaves and finishes the proof.
\end{proof}

Next we show that the constructed triple $(B,F,V)$ corresponds to a smooth point of the Hilbert scheme $\HH_1$ having locally the expected dimension $74$. This is an essential element in  the proof of Theorem \ref{thm:44}. Recall that  the map $\nu_B$ has been defined via (\ref{eq:nuB}).

\begin{proposition}\label{prop:transv8}
The map
$\nu_B\colon H^0(B,N_{B/\GG})\longrightarrow \mathrm{Hom}\bigl(\mathrm{Ker}(\mathrm{det}_V), \mathrm{Coker}(\mathrm{det}_V)\bigr)$ is surjective. It follows that $\HH_1$ is smooth of dimension $74$ around the point $[B\hookrightarrow \GG]$.
\end{proposition} 
\begin{proof}
Recall that the Hilbert scheme $\HH$ is smooth of dimension $84$ at the point $[B\hookrightarrow \GG]$, whereas $\mbox{dim}\bigl(\mathrm{Ker}(\mathrm{det}_V)\bigr)=10$ and $\mbox{dim}\bigl(\mathrm{Coker}(\mathrm{det}_V)\bigr)=1$. Recalling that the form $\eta$ has been defined in (\ref{eq:eta1}) as the annihilator of the hyperplane $\mbox{Im}(\mathrm{det}_V)\subseteq H^0(B,L)$, we shall use the identification $\eta\colon \mathrm{Coker}(\mathrm{det}_V)\stackrel{\cong}\longrightarrow \kk$.  We choose an explicit $10$-dimensional subspace of tangent directions to the Hilbert scheme $$U=\langle \xi_1, \ldots, \xi_{10}\rangle \subseteq H^0(B, N_{B/\GG})\cong T_{[B\hookrightarrow \GG]}(\HH),$$
then show that the $10\times 10$ matrix of the restriction map $\nu_{B|U}$ is nonsingular. 

\vskip 4pt

The first seven tangent directions $\xi_1, \ldots, \xi_7$ we choose correspond to general deformations of the vector bundle $F$ sitting in an extension (\ref{eq:extF}) and of lifts $\widetilde{m}_i\in H^0(B,F)$ of the sections $m_i$ forming a basis of $H^0(B,M)$. Precisely, we introduce the affine parameter space of such tuples, hat is,
$$\mathcal{Y}:=\Bigl\{\bigl(F, \chi\colon H^0(B,M)\otimes \OO_B\rightarrow F\bigr): F\in \mathrm{Ext}^1(M,A), \ \pi_M\circ \chi=\mathrm{ev}_M\Bigr\},$$
where $\mathrm{ev}_M\colon H^0(B,M)\otimes \OO_B\rightarrow M$ is the evaluation map. Since $\mbox{dim } \mathrm{Ext}^1(M,A)=8$, it follows that $\mbox{dim}(\mathcal{Y})=22$. In concrete terms, $\mathcal{Y}$ is the subspace of the vector space $\bigwedge^2 H^0(B,M)\otimes H^0(B,L)\cong H^0(B,L)^{\oplus 21}$ consisting of tuples $$\gamma=(\gamma_{ij})_{0\leq i<j\leq 6}\in H^0(B,L)^{\oplus 21}$$ satisfying the relations  
$(y_a m_i)\cdot \gamma_{jk}-(y_a m_j)\cdot \gamma_{ik}+(y_a m_k)\cdot \gamma_{ij}=0\in H^0(B, L^2)$, for all triples  $0\leq i<j<k\leq 6$ and $a=0,1$. 

\vskip 4pt

There is a natural map $\mathcal{Y}\rightarrow \mbox{Quot}_B(V\otimes \OO_B)$ and an induced tangent map 
\begin{equation}\label{eq:diff_4}
T\bigl(\mathcal{Y}\bigr) \longrightarrow T_{[V\otimes \OO_B\twoheadrightarrow F]}\mathrm{Quot}_B(V\otimes \OO_B)=H^0(B, M_V^{\vee}\otimes F)\cong H^0(B, T_{\GG|B})
\end{equation}
%Concretely, for $\gamma=(\gamma_{ij})_{i<j}$, we associate the alternating form $d_{\gamma}\colon \bigwedge^2 H\rightarrow H^0(B,L)$ defined on the canonical basis $(h_0, \ldots, h_8)$ by setting 
%$$d_{\gamma}(h_0\wedge h_1)=d(h_0\wedge h_{i+2})=d(h_1\wedge h_{i+2})=0, \mbox{ for } i=0, \ldots, 6,$$ and $d_{\gamma}(h_{i+2}\wedge h_{j+2})=\gamma_{ij}$. 

We choose  $\tau_1, \ldots, \tau_7\in H^0(B, T_{\GG|B})$ tangent directions corresponding to images of random choices in $T\bigl(\mathcal{Y}\bigr)$ under the map (\ref{eq:diff_4}) and then $\xi, \ldots, \xi_7\in H^0(B, N_{B/\GG})$ are the corresponding tangent vectors to the Hilbert scheme $\mathfrak{H}$.  

\vskip 4pt

The tangent direction $\xi_8$ corresponds to deforming $W\in G\bigl(5, H^0(B,M)\bigr)$ and is defined as follows. We have $H^0(B,T_{\GG/B})\cong \mbox{Hom}(M_V,F)\subseteq \mbox{Hom}(V\otimes \OO_B,F)$ and we consider the morphism of sheaves $\tau_8\colon M_V\rightarrow F$ as a restriction to $M_V$ of the morphism of sheaves $V\otimes \OO_F\rightarrow F$ defined by assigning $\tau_8(v_2)=h_7$ and $\tau_8(v_j)= 0\in H^0(B,F)$, for $j\neq 2$. Then $\xi_8\in H^0(B, N_{B/\GG})$ is the image of $\tau_8$ under the natural surjection $H^0(B,T_{\GG|B})\twoheadrightarrow H^0(B,N_{B/\GG})$.

\vskip 3pt

Then $\xi_9\in H^0(B,N_{B/\GG})$ corresponds to varying the curve $[B]\in \cM_4$ in its moduli space along a Schiffer variation based at the point $p=(3,4)\in B(\kk)$. Precisely, under the surjection $\partial\colon H^0(B, N_{B/\GG})\rightarrow H^1(B,T_B)\cong H^0(B,\omega_B^2)^{\vee}$ describing the tangent map of the forgetful morphism $\HH \dashrightarrow \cM_4$, we have that $\partial(\xi_9)\colon H^0(B,\omega_B^2)\rightarrow \overline{\kk}$ is the evaluation map at the point $p$. Finally, the vector $\xi_{10}$ is constructed as the image of the morphism $\tau_{10}\colon M_V\rightarrow F$ obtained by setting $\tau_{10}(h_0)= h_2\in H^0(B,F)$ and $\tau_{10}(h_1)=0$ and $\tau_{10}(\widetilde{w}_i)=0$, for $i=0, \ldots, 4$.   

\vskip 4pt

We now check directly that the $10\times 10$ matrix having columns $\nu_B(\xi_1), \ldots, \nu_B(\xi_{10})$ is non-singular, which finishes the proof.

\end{proof}

\subsection{The surjectivity of the twisted Petri map} Via Propositions \ref{prop:transv8} and \ref{prop:transv5} we established that the variety $\cZ$ introduced in Definition \ref{def:d2} has a component of dimension $44$ containing a point $t=(C,E,x+y+z,\tau)$ corresponding via a linkage construction to a smooth triple $(B,F,V)$. To finish the proof of Theorem \ref{thm:bnr2} we show that the same component of $\cZ$ contains an isolated point in its fibre of the forgetful map $\cZ\rightarrow \cM_{16}$. 

\vskip 4pt

We retain all the choices and the notation used in the construction of the triple $(B,F,V)$. We choose a suitable $11$-dimensional subspace $\PP_0$
$$\Lambda=\PP\bigl(\mathrm{Im}(\mathrm{det}_V^{\vee})\bigr)\subseteq \PP_0\subseteq \PP\bigl(\bigwedge^2 V^{\vee}\bigr),$$
and then define a genus $16$ curve $C$ as the curve residual to $B$ in the intersection $\PP_0\cap \GG$.

\vskip 4pt

The vector space $V\cong H^0(B,F)$ has the distinguished basis $(v_0, \ldots, v_6)$ and we denote by $\bigl(e_0:=v_0\wedge v_1, e_1:=v_0\wedge v_2,  \ldots, e_{20}=v_5\wedge v_6\bigr)$ the induced basis of $\bigwedge^2 V$, where the ordered is lexicographical. Note that $e_{11}=v_2\wedge v_3$ is the first element of the basis not involving $v_0$ or $v_1$. The elements $\mathrm{det}_V(e_1), \ldots, \mathrm{det}_V(e_{11})$ form a basis of $\mbox{Im}(\mathrm{det}_V)$, in particular for $j=12, \ldots, 20$, we can write 
$$\mathrm{det}_V(e_j)=\lambda_{j,0}\cdot \mathrm{det}_V(e_1)+\cdots \lambda_{j,10}\cdot \mathrm{det}_V(e_{11}), \ \mbox{ for } \ \ \lambda_{j,a}\in \kk.$$ In particular, for $j=12, \ldots, 20$, we have
$$ k_j:=e_j-\lambda_{j,0}\cdot e_1-\cdots-\lambda_{j,10}\cdot e_{11}\in \mbox{Ker}(\mathrm{det}_V),$$ and in fact it is not hard to show that $\bigl(e_0, k_{12}, \ldots, k_{20}\bigr)$ is a basis of $\mbox{Ker}(\mathrm{det}_V)$. Each $k_j$ corresponds to a hyperplane in $\PP\bigl(\bigwedge^2 V^{\vee}\bigr)$ containing $\Lambda$. We now define the following hyperplane $P_0^{\perp}$ in $\mbox{Ker}(\mathrm{det}_V)$ 
$$
\bigl\langle k_{12}+7e_0, k_{13}+15e_0, k_{14}+4e_0, k_{15}+10e_0, k_{16}+
\\9e_0,k_{17}+11e_0,  k_{18}+6e_0,k_{19}+10e_0, k_{20}+6e_0\bigr\rangle.
$$

The intersection $\PP_0$ of the hyperplanes in $P_0^{\perp}$ is an $11$-dimensional subspace of $\PP\bigl(\bigwedge^2V^{\vee}\bigr)$ containing $\Lambda$. Note that the curve $C$ defined by the relation $$\PP_0\cdot \GG=B+C,$$ is smooth, has genus $16$  and comes equipped with the vector bundle $F_C:=\cQ_{|C}$. Then $\mbox{det}(F_C)\cong \omega_{C}(-D)$, where $D=x+y+z$ is the reduced length $3$ subscheme defined by the relation $\Lambda\cdot \GG=B+D$. Since $\langle C\rangle=\PP_0$ is $11$-dimensional, the determinant map 
$$\mbox{det}_{F_C}\colon \bigwedge^2 V \longrightarrow H^0\bigl(C, \omega_{C}(-D)\bigr)$$
has corank one, in particular, $[C, F, D]$ lies in the subvariety $\cZ_1\subseteq \cZ$ introduced in Definition \ref{def:Z12}. We identify $H^0(C, F_C)\cong V$ and  $\mbox{det}_{F_C}$ with the projection  
$\bigwedge^2 V\rightarrow \frac{\bigwedge^2 V}{P_0^{\perp}}$.

\begin{proposition}\label{prop:transv9}
The following holds:
\begin{enumerate}
\item The vector bundle $F_C$ is stable and $h^0(C,F_C)=7$.
\item The twisted Petri map $\mu_{F_C,D}$ has rank $42$.
\end{enumerate}
\end{proposition}
\begin{proof}
For the first part we use Lemma \ref{lemma:resonance} and we show that the vector bundle $F_C$ has trivial resonance. Letting $\PP_0^{\perp}$ be the projectivization of $P_0^{\perp}$, this amounts to  
$$\PP_0^{\perp}\cap G\bigl(2, \bigwedge^2 V\bigr)=\emptyset,$$ which follows by inspection from our choice of $P_0^{\perp}$. 

\vskip 4pt

For the second part, we recall the identification $H^0(C, F_C)\cong \langle v_0, \ldots, v_6\rangle$. Setting $L_C=\omega_{C}(-x-y-z)$,  let $z_j:=\mathrm{det}_{F_C}(e_j)\in H^0(C, L_C)$ for $j=0, \ldots, 11$. To study the map $\mu_{F_C, D}$, we first get to grips with the usual Petri map 
$$\mathfrak{pe}_{F_C}\colon H^0(C, F_C)\otimes H^0(C, F_C^{\vee}\otimes \omega_{C})\longrightarrow H^0\bigl(C, \mathrm{End}(F_C)\otimes \omega_{C}\bigr).$$
To get bases for the spaces appearing in this map, we consider the degree $27$ cover $f\colon C\rightarrow \PP^1$, given by  $\langle z_0, z_1\rangle \subseteq H^0(C, L_C)$. Denoting by $S:=\kk[s,t]:=\Gamma(\PP^1, \OO_{\PP^1}(1))$ the homogeneous coordinate ring of $\PP^1$, we consider the homogeneous coordinate ring 
$$\Gamma_C(L_C):=\bigoplus_{n\geq 0} H^0(C,L_C^n)$$ viewed as graded $S$-module via the cover $f$. Then we show by Riemann-Roch that 
$$\Gamma_C(L_C)\cong S \oplus S(-1)^{\oplus 11}\oplus S(-2)^{\oplus 14}\oplus S(-3).$$ 
Similarly we form the graded $S$-module $\Gamma_C(F_C)=\bigoplus_{n\geq 0} H^0(C, F_C\otimes L_C^n)$. Since one has $f_*F_C\cong \OO_{\PP^1}^{\oplus 7}\oplus \OO_{\PP^1}(-1)^{\oplus 37}\oplus \OO_{\PP^1}(-2)^{10}$,  we then have 
$$\Gamma_C(F_C)\cong S^{\oplus 7}\oplus S(-1)^{\oplus 37}\oplus S(-2)^{\oplus 10},$$ 
and we choose a  set of generators of $\Gamma_C(F_C)$ compatible with this decomposition and given by $v_0, \ldots, v_6\in H^0(C,F_C)$,  the basis we started with,  $f_0,\ldots, f_{36}\in H^0(C,F_C(1))$ and $g_0,\ldots, g_{9}\in H^0(C,F_C(2))$. 

\vskip 4pt

We then have the identification $H^0(C, F_C^{\vee}\otimes \omega_C)\cong \mathrm{Hom}_S\bigl(\Gamma_C(F_C), S(-2))_0\cong \kk^{10}$. We choose an explicit basis $(w_0, \ldots, w_9)$ of $H^0(C,F_C^{\vee}\otimes \omega_C)$ by setting $$w_i\colon \Gamma_C(F_C)\longrightarrow S(-2)$$
 defined by $w_i(g_a)=\delta_{ia}$, for $i,a=0, \ldots, 9$  and $v_0, \ldots, v_6, f_0, \ldots, f_{36}\in \mbox{Ker}(w_i)$.

\vskip 3pt

We pass from $\mathfrak{pe}_{F_C}$ to $\mu_{F_C,D}$ and a direct calculation shows that $\mbox{rk}\bigl(\mu_{F_C,D}\bigr)=42$.
\end{proof}

The reason why the curve $C$ obtained via a linkage construction from a genus $4$ curve could not be used to show that the twisted Petri map $\mu_{F_C,D}$ has maximal rank was that $\mbox{dim}\langle C\rangle=11$. When we move $C\subseteq \GG$ in such a way that its Pl\"ucker span becomes $12$-dimensional, it will turn out that this phenomenon disappears:  

\vskip 5pt

\noindent \emph{End of the proof of Theorem \ref{thm:bnr2}.} Recall that we have constructed the point  $[C\hookrightarrow \GG]\in \mathfrak{H}_1$ corresponding to a smooth curve of genus $16$ and we retain the notation from Proposition \ref{prop:transv9}. We take the embedding $C\stackrel{|L_C|}\longrightarrow \PP^{12}$ and consider the homogeneous coordinate ring $\Gamma_{C}(L_C)$. Setting $R:=\mbox{Sym } H^0(C,L_C)$, we have the following minimal resolution of $\Gamma_C(L_C)$ as a graded $R$-module
$$\cdots\longrightarrow R(-3)^{\oplus 287}\longrightarrow R(-2)^{\oplus 52}\longrightarrow S\longrightarrow \Gamma_C(L_C)\longrightarrow 0.$$
 Moreover, the ideal $I_{C/\PP^{12}}$ is generated by $52$ quadrics $Q_1, \ldots, Q_{52}$, of which $Q_1, \ldots,  Q_{35}$ are the restrictions to $\PP_0$ of the Pl\"ucker quadrics defining the Grassmannian $\GG$. It can be computed with \emph{Macaulay} that $L_C$ satisfies the Green-Lazarsfeld property $(N_3)$. For every quadric $Q\in I_2(C,L_C)$ and for every tangent vector $\zeta\in H^0(C, N_{C/\PP^{12}})$, we denote by $\zeta(Q)\in H^0(C,L_C^2)$ the corresponding element under the morphism  
$$I_2(L_C)\otimes \OO_C \twoheadrightarrow N_{C/\PP^{12}}^{\vee}\otimes L_C^2.$$ 

\vskip 4pt

To obtain a suitable deformation of $C$, we deform the ideal $I_{C/\PP^{12}}$, or equivalently we choose a tangent direction $\zeta\in H^0(N_{C/\PP^{12}})$ such that whenever we have a syzygy 
$\sum_{i=1}^{52} \ell_i\otimes Q_i\in \mathrm{Ker}\bigl\{ H^0(C,L_C)\otimes I_2(L_C)\rightarrow I_3(L_C)\bigr\}$, then  necessarily 
$$\sum_{i=1}^{52} \ell_i\cdot \zeta(Q_i)=0\in H^0(C, L_C^3).$$ 
Such a $\zeta$ does give rise to a deformation $C_t\subseteq \PP^{12}\subseteq \GG$ (because of the compatibility with the Pl\"ucker relation, the deformation $C_t$ stays in $\GG$) and to an induced map
\begin{equation}\label{eq:diff_normal}
\Theta_{\zeta}\colon \mathrm{Ker}\bigl(\mu_{F_C,D}\bigr)\longrightarrow \mathrm{Coker}\bigl(\mu_{F_C,D}\bigr).
\end{equation}

Proposition \ref{prop:transv9} yields $\mbox{dim}\bigl(\mathrm{Ker}(\mu_{F_C,D})\bigr)=7$ and $\mbox{dim}\bigl(\mathrm{Coker}(\mu_{F_C,D})\bigr)=6$.
By a random search, we find one example of $\zeta$ such that the map $\Theta_{\zeta}$ has rank $6$, shows that on the deformation $C_t\subseteq \GG$, denoting by $F_t=\cQ_{C_t}$ we have $\mu_{F_t, D_t}$ has maximal rank equal to $48$, which proves that the forgetful map $\mathfrak{H}\rightarrow \mm_{16}$ fills-up a divisor, which finishes the proof. 
\hfill $\Box$

\section{The Proof of Theorem \ref{main_thm}}
In this final Section we complete the proof that the canonical divisor $K_{\mm_{16}}$ is not pseudoeffective. We do this in two parts and first we determine the intersection of the non-abelian Brill-Noether divisor $\dd2$ with the test curve $\Gamma_2$ constructed in the course of proving Theorem \ref{thm:bnr2}.

\vskip 4pt

We begin with a few preliminaries. We start with a general curve $[C]\in \dd2$ and with a very ample vector bundle $F$ on $C$, with $\mbox{det}(F)\cong \omega_C(-D)$, where $D=x+y+z$ is a reduced divisor such that the canonical surface $X=\Pi \cdot \GG$ defined by (\ref{eq:X}) is smooth. We set $V:=H^0(C,F)$, thus $\GG=G(2,V^{\vee})$. The pair $(C, F)$ corresponds to a point in $\cZ\setminus \cZ_1$, in particular the map $\mbox{det}\colon \bigwedge^2 V\rightarrow H^0\bigl(C, \omega_C(-x-y-z)\bigr)$ is surjective. We may also assume that the twisted Petri map $\mu_{F,D}$ defined by (\ref{eq:twisted}) is surjective, which immediately implies that also the usual Petri map 
$$\mathfrak{pe}_F\colon H^0(C,F)\otimes H^0(C, F^{\vee}\otimes \omega_C)\longrightarrow H^0\bigl(C, \mbox{End}(F)\otimes \omega_C\bigr)$$  is surjective. 
Note that $h^0(C, F^{\vee}\otimes \omega_C)=10$.

\vskip 4pt

Let $\mathfrak{M}$ be the local moduli space classifying polarized surfaces $\bigl(X', \OO_{X'}(C')\bigr)$, where  $\omega_{X'}^2=42$, $(C')^2=3$ and $C'\cdot \omega_{X'}=27$. Precisely, $\mathfrak{M}$ is the subvariety of the deformation space $\mathrm{Def}(X)$ parametrizing deformations $X'$ of $X$ such that $c_1\bigl(\OO_X(C)\bigr)$ remains a class of type $(1,1)$. In order to describe the tangent space of $\mathfrak{M}$, let $\mathfrak{A}_{C}$ be the \emph{Atiyah bundle} sitting in an extension $0\rightarrow \OO_X\rightarrow \mathfrak{A}_C\rightarrow T_X\rightarrow 0$, corresponding to the extension class $c_1\bigl(\OO_X(C)\bigr)\in H^1(X,\Omega_X^1)\cong \mbox{Ext}^1(T_X,\OO_X)$ and consider the induced cohomology exact sequence
$$0\longrightarrow H^1(X, \mathfrak{A}_C)\longrightarrow H^1(X,T_X)\xrightarrow{\qquad \cup \ c_1(\OO_X(C)) \qquad}  H^2(X,\OO_X).$$
Then, we have the following well-known description of the tangent space:
\begin{equation}\label{eq:kuranishi}
T_{[X, \OO_X(C)]}\bigl(\mathfrak{M}\bigr)\cong \mbox{Ker}\bigl\{\cup\ c_1\bigl(\OO_X(C)\bigr)\colon H^1(X,T_X)\longrightarrow H^2(X,\OO_X)\bigr\}.
\end{equation}

Note that $h^1(X,T_X)=10\chi(X,\OO_X)-2\omega_X^2=56$ and $h^2(X,\OO_X)=13$. Let $\mathfrak{B}$ be the parameter space of pairs $(X,C)$, where $[X, \OO_X(C)]\in \mathfrak{M}$ and $C\in |\OO_X(C)|$. We have  morphisms
$$\xymatrix@R=2pc@C=3pc{
\mathfrak{B}\ar[d]_{q}  \ar[r]^{\mathfrak{m}}& \mm_{16}, \\
\mathfrak{M} \\
                 }
$$    
where the forgetful map $q$ is a $\PP^1$-bundle and $\mathfrak{m}\bigl([X,C]\bigr):=[C]$ is the moduli map. With this notation, if $R:=q^{-1}\bigl([X, \OO_X(C)]\bigr)\subseteq \mathfrak{B}$, then  $\Gamma_2:=\mathfrak{m}_*(R)$. 

\vskip 3pt

For $(X,C)$ as above, we denote by $T_X(-\mathrm{log}\ C)$ the vector bundle of vector fields on $X$ tangent to $C$. The first order deformations of the pair $(X,C)$ are classified by $H^1\bigl(X,T_X(-\mathrm{log}\ C)\bigr)$, whereas the obstructions lie in $H^2\bigl(X,T_X(-\mathrm{log}\ C)\bigr)$, see \cite[\S1]{B}. More generally, if $C'\subseteq X$ is a nodal curve, then the deformations are controlled by the hypercohomology of the complex $[T_X\rightarrow N_{C/X}]$.  In particular, we have an exact sequence describing the differential of the map $\mathfrak{m}$, for every nodal curve $C\in \bigl|\OO_X(C)\bigr|$:
\begin{equation}\label{eq:def_mu}
H^1(X, T_X(-C))\longrightarrow T_{[X,C]}(\mathfrak{B})\stackrel{d \mathfrak{m}}\longrightarrow T_{[C]}(\mm_{16})\longrightarrow H^2(X,T_X(-C)).
\end{equation}

\begin{proposition}\label{prop:calculations}
For $(C,F,X)$ as above, the following statements hold:
\begin{enumerate}
\item $h^0(X, T_{\GG|X})=48,\ \ h^1(X, T_{\GG|X})=0\  \mbox{ and }  \ h^2(X, T_{\GG|X})=8$.
\item $H^1(X,T_X(-C))=0$.
\item $h^1(X, T_X(-\mathrm{log}\ C))=44$.
\end{enumerate}
\end{proposition}
\begin{proof}
We denote by $W:=\mbox{Ker}(\mathrm{det})\subseteq H^0(\GG,\OO_G(1))$ the $8$-dimensional system of hyperplanes containing $X$, then write down the Koszul resolution of  $X\subseteq \GG$:
$$0\longrightarrow \bigwedge^8 W\otimes \OO_{\GG}(-8)\longrightarrow \cdots \longrightarrow \bigwedge^2 W\otimes \OO_{\GG}(-2)\longrightarrow W\otimes \OO_{\GG}(-1)\longrightarrow \OO_X\longrightarrow 0.$$ 
After twisting this complex with $T_{\GG}\cong \mbox{Hom}(\mathcal{S},\cQ)$, where $\mathcal{S}$ and $\cQ$ are the universal sub- and quotient bundles on $\GG$, the cohomology groups $H^{q-p}(X, T_{\GG|X})$ are computed by the hypercohomology spectral sequence of the Koszul complex having terms $E_1^{-p,q}=\bigwedge^p W\otimes H^q\bigl(\GG,T_{\GG}(-p)\bigr)$. These are computed via the Bott-Weyl theorem and the only non-zero terms are 
$$E_1^{0,0}\cong \frac{\mathrm{Hom}(V,V)}{\mathbb C\cdot \mathrm{id}_V}\cong \mathfrak{sl}(V) \ \mbox{ and } E_1^{-7,9}\cong \bigwedge^7 W\otimes \bigl(\mathrm{det}(V^{\vee})\bigr)^{2},$$
where $H^9\bigl(\GG, T_{\GG}(-7)\bigr)\cong H^1(\GG, \Omega_{\GG}^1)^{\vee}$.
The first statement now follows and we have  
$$H^0(X, T_{\GG|X}) \cong \mathfrak{sl}(V) \ \ \  \mbox{ and } \ \ \  H^2(X, T_{\GG|X})\cong \bigwedge^7 W\otimes \bigl(\mathrm{det}(V^{\vee})\bigr)^{2}.$$
In particular, $h^0(X,T_{\GG|X})=48$ and $h^2(X,T_{\GG|X})=8$.

\vskip 4pt

We introduce the kernel bundle of $F$ via the following exact sequence on $C$
\begin{equation}\label{eq:kernel-bundle-F}
0\longrightarrow M_F\longrightarrow H^0(C,F)\otimes \OO_C\longrightarrow F\longrightarrow 0
\end{equation}
and we have the identification $T_{\GG|C}\cong M_F^{\vee}\otimes F$. 
Apply $\mathcal{H}om(-, F)$ to the sequence (\ref{eq:kernel-bundle-F}), take cohomology, then use that $F$ is stable, thus also simple, to find the exact sequence
$$0\longrightarrow \mathfrak{sl}(V)\longrightarrow H^0(C, T_{\GG|C})\longrightarrow \bigl(\mathrm{Coker}(\mathfrak{pe}_F)\bigr)^{\vee}\longrightarrow 0.$$
Since the Petri map $\mathfrak{pe}_F$ is surjective, we obtain $h^0(C,T_{\GG|C})=\mbox{dim } \mathfrak{sl}(V)=48$, in particular we also have  $H^0(X, T_{\GG|X})\stackrel{\cong}\longrightarrow  H^0(C,T_{\GG|C})$, therefore $H^1(X, T_{\GG|X}(-C))=0$.

\vskip 3pt

From the exact sequence $0\rightarrow T_X\rightarrow T_{\GG|X}\rightarrow \OO_X(1)^{\oplus 8}\rightarrow 0$, after twisting with $\OO_X(-C)$ and recalling that $h^0(X,\omega_X(-C))=h^1(X, \omega_X(-C))=0$, we obtain $H^1(X, T_X(-C))= 0$. 

\vskip 4pt

To deal with statement (iii), we take cohomology in the following exact sequence 
$$0\longrightarrow T_X(-C) \longrightarrow T_X(-\mathrm{log}\ C)\longrightarrow T_C \longrightarrow 0.$$
Since $H^1(X, T_X(-C))=0$, we have $H^1\bigl(X, T_X(-\mathrm{log}\ C)\bigr)\hookrightarrow H^1(C,T_C)$. Furthermore, from the exact sequence $0\rightarrow T_X(-\mathrm{log}\ C)\rightarrow T_X\rightarrow N_{C/X}\rightarrow 0$ we obtain
$$h^1\bigl(X, T_X(-\mathrm{log}\ C)\bigr)=1+\mathrm{dim}\ \mathrm{Ker}\bigl\{H^1(X,T_X)\longrightarrow H^1(N_{C/X})\bigr\}.$$
In particular $44\leq h^1\bigl(X, T_X(-\mathrm{log}\ C)\bigr)\leq 45$. The possibility $h^1\bigl(X, T_X(-\mathrm{log}\ C)\bigr)=45$ can be ruled out via the identification $T_{[X,C]}(\mathfrak{B})\cong H^1\bigl(X, T_X(-\mathrm{log}\ C)\bigr)$ since we have already established that $\dd2$ is a divisor in $\mm_{16}$.
\end{proof}

\begin{remark}\label{rmk:constant}
The properties $H^1(X,T_X(-C))=0$ and $h^1\bigl(X,T_X(-\mathrm{log}\ C)\bigr)=44$ which were established in Propsoition \ref{prop:calculations} for \emph{a} smooth curve $C\in \bigl|\OO_X(C)\bigr|$, obviously propagate to all curves in this linear system, that is, throughout the curve $\Gamma_2\subseteq \mm_{16}$.
\end{remark}

\begin{proposition}\label{prop:minus1} One has $\Gamma_2\cdot \dd2\geq -1$.
\end{proposition}
\begin{proof}
We fix a general pair $[X,C]\in \mathfrak{B}$ as above and set $R=q^{-1}\bigl([X, \OO_X(C)]\bigr)\cong \PP^1$ and then $\Gamma_2=\mathfrak{m}_*(R)$. One has the exact sequence of normal bundles
$$0\longrightarrow N_{R/\mathfrak{B}}\longrightarrow N_{R/\mm_{16}}\longrightarrow \mathfrak{m}^*N_{\dd2/\mm_{16}}\longrightarrow 0.$$
The conditions singled out in Remark \ref{rmk:constant} hold along $R$, therefore $\mathfrak{B}$ is smooth of dimension $44$ along $R$ and via (\ref{eq:def_mu}) the differential of $\mathfrak{m}$ is injective. Since $q$ is a smooth morphism along $R$, we have $N_{R/\mathfrak{B}}\cong T_{[X, \OO_X(C)]}(\mathfrak{M})\otimes \OO_R\cong \OO_R^{\oplus 43}$. From the exact sequence $$0\longrightarrow T_R\longrightarrow \mathfrak{m}^*T_{\mm_{16}}\longrightarrow N_{R/\mm_{16}}\longrightarrow 0,$$ we compute via (\ref{eq:int3}) that $\mbox{deg}\bigl(N_{R/\mm_{16}}\bigr)=-2-\Gamma_2\cdot K_{\mm_{16}}=-2+1=-1$, in particular $\Gamma_2\cdot \dd2\geq \mbox{deg}\bigl(\mathfrak{m}^*N_{\dd2/\mm_{16}}\bigr)=-1$.
\end{proof}
 
\begin{remark}\label{rmk:conductor}
Likely one has $\Gamma_2\cdot \dd2=-1$. Proposition \ref{prop:minus1} leaves however open the  possibility that a second branch of the divisor $\dd2$ intersects the curve $\Gamma_2$ at finitely many points.
\end{remark} 

\begin{theorem}\label{thm:slope-ruleout}
One has $s(\dd2)>\frac{13}{2}$.
\end{theorem}
\begin{proof}
We write $[\dd2]=a\lambda-b_0\delta_0-\cdots -b_8\delta_8\in CH^1(\mm_{16})$. Assuming $\frac{a}{b_0}=\frac{13}{2}$, by applying \cite[Corollary 1.2]{FP}, we obtain $b_i\geq b_0$ for $i=1, \ldots, 8$, therefore $s(D)=\frac{a}{b_0}=\frac{13}{2}$. From Propositions \ref{prop:minus1} and \ref{prop:pencild2}, we obtain that $29a-189b_0\geq -1$. If  $29a-189b_0\geq 0$, then $s(\dd2)\geq \frac{189}{29}>\frac{13}{2}$, which is a contradiction. Thus, we are left with the case $29a-189b_0=-1$, in which case necessarily $a=13$ and $b_0=2$ and therefore $K_{\mm_{16}}-\dd2$ is a (possibly empty) effective combination of the boundary classes $\delta_1, \ldots, \delta_8$ and, in particular, $H^0\bigl(\mm_{16}, K_{\mm_{16}}\bigr)\neq 0$. 

\vskip 4pt

We assume by contradiction that $[\dd2]=13\lambda-2\delta_0-\sum_{i=1}^8 b_i\delta_i$. We  pick a general $2$-pointed curve $[C,x,y]\in \mm_{15,2}$, choose $A\in W^1_9(C)$ such that $h^0(C, A(-x-y))=1$ and construct the sweeping rational curve $\Gamma\subseteq \Delta_0$ from Theorem \ref{sweep} and which passes through the point $[Y:=C/x\sim y]$. Since $[C,x,y]$ has been chosen generically, clearly $[Y]\notin \dd2$. From Proposition \ref{numbers} we compute $\Gamma\cdot \dd2=0$, therefore $\Gamma\cap \dd2=\emptyset$. 

\vskip 4pt

However this conclusion is impossible. We choose a point $r\in C$ such that $h^0(C,A(-x-y-r))\neq 0$. Setting, as before, $L=\omega_C\otimes A^{\vee}\in W^6_{19}(C)$, we denote by $\nu\colon C\rightarrow Y$ the normalization map. We let $M$ and $N$ be the line bundles on $Y$, with $h^0(Y,M)=2$ and $h^0(Y,N)=5$ such that $\nu^*(M)\cong A$ and $\nu^*N\cong L(-r)$. We then consider a general  vector bundle $\widetilde{F}$ on $C$ lying in an extension  
$$0\longrightarrow M\longrightarrow F \longrightarrow N\longrightarrow 0,$$
such that $h^0(Y,F)=7$. Such an extension exists if and only if the coboundary map $H^0(Y,N)\rightarrow H^1(Y,M)$ is zero. Since $\nu^*(\omega_Y)\cong \omega_C(x+y)$, it follows that $F$ can be written as limit of vector bundles with special determinant of degree $27$ on smooth curves of genus $16$, in particular $\Gamma\cdot\dd2\neq \emptyset$, which finishes the proof.

\end{proof}

\end{document}